\documentclass[12pt,oneside]{article}
\usepackage{amsmath,amstext,amssymb,amsthm}
\usepackage{hyperref}
\usepackage{color}
\usepackage{dsfont}
\usepackage{tikz}
\usepackage{mathrsfs}
\usetikzlibrary{arrows}
\usepackage{graphics}
\usepackage{soul}

\theoremstyle{plain}
\newtheorem{theorem}{Theorem}[section] 
\newtheorem{lemma}[theorem]{Lemma}
\newtheorem{proposition}[theorem]{Proposition}
\newtheorem{corollary}[theorem]{Corollary}
 \theoremstyle{definition}
\newtheorem{defn}{Definition}

\newtheorem{example}{Example} \theoremstyle{remark}
\newtheorem*{remark*}{Remark} \newtheorem{remark}{Remark}

\newcommand{\R}{\mathbb{R}}
\newcommand{\Rd}{{\R^{d}}}
\newcommand{\NN}{\mathbb{N}}
\newcommand{\ZZ}{\mathbb{Z}}
\newcommand{\ind}{\mathds{1}}
\renewcommand{\leq}{\leqslant}
\renewcommand{\geq}{\geqslant}

\newcommand{\GZL}[1]{G^{Z,{#1}}}
\newcommand{\GZLN}[2]{G^{Z,{#1},{#2}}}
\newcommand{\GZLn}[1]{\GZLN{{#1}}{n}}

\def \GZln{\GZLn{\lambda}}

\def \KK{\mathbb{K}}
\def \pK{\mathcal{K}}
\def \PP{\mathbb{P}}
\def \EE{\mathbb{E}}
\def \Hyp{(H0)}

\def\({\left(}
\def\){\right)}
\def\[{\left[}
\def\]{\right]}
\def\<{\langle}
\def\>{\rangle}

\newcommand{\WUSC}[3]{\textrm{WUSC}(#1,#2,#3)}
\newcommand{\WLSC}[3]{\textrm{WLSC}(#1,#2,#3)}
\newcommand{\LL}{{\eta}}
\newcommand{\lC}{{\underline{c}}}
\newcommand{\uC}{{\overline{C}}}
\newcommand{\la}{{\underline{\alpha}}}
\newcommand{\ua}{{\overline{\alpha}}}
\newcommand{\lt}{{\underline{\theta}}}
\newcommand{\ut}{{\overline{\theta}}}

\definecolor{ks}{rgb}{0.7,0.1,0.2}

\definecolor{ksg}{rgb}{0.1,0.4,0.1}

\title{Kato classes for L\'{e}vy processes}
\author{Tomasz Grzywny and Karol Szczypkowski \thanks{Authors' affiliations and emails: 
  Wroc\l{}aw University of
Science and Technology, 
Wyb. Wyspia\'{n}skiego
27, 50-370 Wroc\l{}aw, Poland, 
Universit{\"a}t Bielefeld, Postfach 10 01 31, D-33501 Bielefeld, Germany
tomasz.grzywny@pwr.edu.pl, karol.szczypkowski@math.uni-bielefeld.de, karol.szczypkowski@pwr.edu.pl.
}}

\begin{document}
\maketitle
\begin{abstract}
We prove that the definitions of the Kato class through the semigroup and through the resolvent of the L{\'e}vy process 
in
$\Rd$ coincide
if and only if $0$ is not regular for $\{0\}$. If 0 is regular for $\{0\}$ then we describe both classes in detail. We also give an analytic reformulation of these results by means of the characteristic (L{\'e}vy-Khintchine) exponent of the process. The result applies to the time-dependent (non-autonomous) Kato class. As one of the consequences we obtain a simultaneous time-space smallness condition equivalent to the Kato class condition given by the semigroup.
\end{abstract}

\section{Introduction}

The Kato class plays an important role in the theory of stochastic processes
and in the theory of pseudo-differential operators
that emerge as generators of stochastic processes.
The definition of the Kato class may differ according to the underlying probabilistic or analytical problem.
In~the first case the primary
definition of the Kato condition is
\begin{align}\label{start}
\lim_{t\to 0^+} \left[\sup_{x} \EE^x \left( \int_0^t |q(X_u)|\,du\right)\right]=0\,.
\end{align}
Here $q$ is a Borel function on the state space of the process $X=(X_t)_{t\geq 0}$.
As shown in \cite[section 3.2]{MR1329992} through the Khas'minskii Lemma the condition yields sufficient local regularity of the corresponding Schr{\"o}dinger (Feynman-Kac) semigroup 
$$
\widetilde{P}_tf(x)=\EE^x \left[\exp \left(-\int_0^t q(X_u)\,du\right)f(X_t)\right].
$$ 
In particular, the existence of a density, strong continuity or strong Feller property are inherited under \eqref{start} from 
properties of the
original semigroup $P_tf(x)=\EE^x f(X_t)$ (for details and further results see \cite[Theorems~3.10--3.12]{MR1329992}).
Moreover, if we denote by $L$ the generator of $(P_t)_{t\geq 0}$, 
we expect the semigroup $(\widetilde{P}_t)_{t\geq 0}$ to correspond 
to $L-q$ and to allow for the analysis of the Schr{\"o}dinger operator $H=-L+q$ (\cite{MR1772266}). A 
fact that  the Schr{\"o}dinger operator is essentially self-adjoint and has   bounded and continuous eigenfunctions is another consequence of \eqref{start}, see~\cite{MR1054115}, \cite{MR670130} and \cite{KL}.
Applications of \eqref{start} to quadratic forms of Schr{\"o}dinger operators   are also known and we describe them shortly
after Proposition~\ref{gen_KK_equiv}.

The condition \eqref{start} can be understood as a smallness condition with respect to time. The alternative definition of the Kato condition is given by the following space smallness,
\begin{align}\label{starteqiv}
\lim_{r\to 0^+}\left[ \sup_{x}\EE^x \left( \int_0^{\infty} e^{-\lambda u} \ind_{B(x,r)}(X_u)|q(X_u)|\,du\right)\right]=0\,,
\end{align}
for some $\lambda>0$ (equivalently for every $\lambda>0$; see Lemma~\ref{lem:consistency}).

In this paper we
obtain a precise description of the equivalence of \eqref{start} and \eqref{starteqiv} for L{\'e}vy processes in $\Rd$,  $d\in\mathbb{N}$.
In order to formulate the result we recall that a point $x\in\Rd$ is said to be {\it regular} for a Borel set $B\subseteq\Rd$ if 
$$\PP^x(T_B=0)=1\,,$$
where 
$T_B=\inf \{t>0: X_t\in B\}$ is the first hitting
time of $B$.
\begin{theorem}\label{thm:Gen_1}
Let $X$ be a L\'{e}vy process in $\Rd$. The conditions \eqref{start} and \eqref{starteqiv} are NOT equivalent if~and~only~if $0$ is 
regular for $\{0\}$.
\end{theorem}
\noindent
Complete and direct descriptions of \eqref{start} and \eqref{starteqiv} in the case of the {\it  compound Poisson process} are given in Proposition~\ref{calInKato}.
When $X$ is not a compound Poisson process and {\it $0$ is regular for $\{0\}$}  we fully describe \eqref{start} and \eqref{starteqiv} in Theorem~\ref{thm:0regH0} and~\ref{thm:WH0_0reg}.
To 
move
right away
to Section~\ref{sec:thm} we 
recommend
to read Definition~\ref{def:Kato} and Section~\ref{sec:scheme} first.
In Section~\ref{sec:scheme}
the reader will also find
analytic characterization of 
the situation when
$0$ is regular for $\{0\}$.

In \cite[Theorem~III.1]{MR1054115}
Carmona, Masters and Simon
declare
that
\eqref{start} 
can be 
expressed
by \eqref{starteqiv}
under additional
assumptions
on the transition density of the L{\'e}vy process.
However, the general 
equivalence
of (i) 
and 
(iii) from \cite[Theorem III.1]{MR1054115} that is claimed therein
does not hold. 
As we show in Theorem~\ref{thm:0regH0} it fails for the Brownian motion in $\R$ and for those one-dimensional unimodal L{\'e}vy processes for which $\{0\}$ is not polar. 
Recall that
a Borel set $B\subseteq \Rd$ is called {\it polar} if 
$$\PP^x (T_B=\infty)=1\qquad \mbox{for all}\quad x\in\Rd\,.$$
For example the function  $q(x)=\sum_{k=1}^{\infty} 2^{k} \ind_{(k,k+2^{-k})}(x)$ satisfies (i), but fails to satisfy (iii) in \cite[Theorem III.1]{MR1054115} for such processes.
The paper \cite{MR1054115} was very influential and the mistake reappears in the literature. For~instance (1) and (3) of \cite[Proposition 4.5]{MR2944475} are {\it not} equivalent in general.

The special character of the one-dimensional case can also be seen in \cite[Remark 3.1]{MR2345907}. In \cite[Definition 3.1 and 3.2]{MR2345907} the authors discuss the Kato class of measures for symmetric Markov processes admitting upper and lower estimates of transition density with additional integrability assumptions, see \cite[Theorem 3.2]{MR2345907}.

Theorem \ref{thm:Gen_1} allows also for results on the time-dependent Kato class for L{\'e}vy processes~in~$\Rd$. Such a class is used for instance in \cite{MR1360232}, \cite{MR1488344}, \cite{MR2457489}, \cite{MR3200161}, \cite{BBS}. See \cite{MR1687500} for a wider discussion of the Brownian motion case, c.f. \cite[Theorem 2]{MR1687500}.
 
\begin{corollary}\label{cor:time}
Let $X$ be a L{\'e}vy process in $\Rd$. For $q\colon \R\times\Rd \to \R$ we have
\begin{align}\label{start_time}
\lim_{t\to 0^+} \left[\sup_{s,x} \EE^x \left( \int_0^t |q(s+u,X_u)|\,du\right)\right]=0\,,
\end{align}
if and only if
\begin{align}\label{start_timeeqiv}
\lim_{r\to 0^+}\left[ \sup_{s,x}\EE^x \left( \int_0^r  \ind_{B(x,r)}(X_u)|q(s+u,X_u)|\,du\right)\right]=0\,.
\end{align}
\end{corollary}

\noindent
See Section~\ref{sec:thm} for the proof.
If one uses Corollary \ref{cor:time} for time-independent $q$, i.e., let $q\colon \Rd\to\R$ and put $q(u,z)=q(z)$, then the quantity in \eqref{start_time} coincides with \eqref{start} and we obtain the following reinforcement of \eqref{start} to a time-space smallness condition.
\begin{theorem}
Let $X$ be a L{\'e}vy process in $\Rd$. Then \eqref{start} holds if and only if
\begin{align}\label{start_strong}
\lim_{r\to 0^+}\left[ \sup_{x}\EE^x \left( \int_0^r  \ind_{B(x,r)}(X_u)|q(X_u)|\,du\right)\right]=0\,.
\end{align}
\end{theorem}
\noindent
In view of the equivalence of \eqref{start} and \eqref{start_strong} for every L{\'e}vy process (see Proposition \ref{gen_KK_equiv} for other description of \eqref{start} true for Hunt processes) these conditions should be compared with \eqref{starteqiv} by its alternative form provided by  Proposition~\ref{pK_alternat} in a generality of a Hunt process, i.e., 
\begin{align}\label{starteqiv_alt}
\lim_{r\to 0^+}\left[ \sup_{x}\EE^x \left( \int_0^t \ind_{B(x,r)}(X_u)|q(X_u)|\,du\right)\right]=0\,,
\end{align}
for some (every) fixed $t>0$. The closeness or possible differences between \eqref{start} and \eqref{starteqiv} are now more evident for L{\'e}vy processes through \eqref{start_strong} and \eqref{starteqiv_alt}.

The variety of conditions we point out is due to possible applications where one can choose a suitable version according to the knowledge about the process and derive a clear {\it analytic description} of the Kato condition \eqref{start}.
See also Theorem~\ref{prop:G_to_zero} and Theorem~\ref{prop:G_unb_b} for other conditions.
For~instance, in  Example~\ref{example:1} we apply Theorem~\ref{thm:Gen_1} and we make use of \eqref{starteqiv_alt}. On the other hand, by Theorem~\ref{thm:Gen_1} and \eqref{starteqiv} we obtain that for a large class of subordinators \eqref{start} is equivalent to
\begin{align}\label{eq:sub_Kato}
\lim_{r\to 0^+}\sup_{x\in\R}\int_0^{r} |q(z+x)| \frac{\phi'(z^{-1})}{z^2\phi^2(z^{-1})}\,dz=0\,,
\end{align}
where $\phi$ is the Laplace exponent of the subordinator. See Section \ref{sec:sub} for details. This is also usual that from \eqref{starteqiv} and \eqref{starteqiv_alt} one learns,
like through \eqref{eq:sub_Kato},
about acceptable singularities of $q$.
Schr{\"o}dinger perturbations of subordinators 
are interesting since they exhibit peculiar properties
that indicate complexity of the matter.
For instance, 
we easily see that if $q$ is bounded, then
$\widetilde{P}_tf(x)\leq c_{N} P_tf(x)$ for every $t\in(0,N]$, $x\in\R$, $f\geq 0$.
On the other hand,
if $-q\geq 0$ is time-independent
and the above inequality holds
for some $N>0$ 
on the level of densities,
then necessarily $q\in L^{\infty}(\R)$ (see \cite[Corollary~3.4]{BBS}).
Nevertheless, perturbation techniques yield an upper bound
by means of an auxiliary density for (unbounded) $q$
from the Kato class if an appropriate 4G inequality 
for the transition density of the subordinator holds (see \cite[Proposition~2.4]{BBS}).
Generators of subordinators generalize
fractional derivative operators that
are used in statistical physics 
to model anomalous subdiffusive dynamics
(see \cite{PhysRevLett.105.170602}).

A discussion of analytic counterparts of \eqref{start} should contain the fundamental example of the standard Brownian motion in $\Rd$, $d\in\NN$. The famous result of Aizenman and Simon \cite[Theorem 4.5]{MR644024} says that in this case \eqref{start} is equivalent to
\begin{align}
\lim_{t\to 0^+} \left[\sup_{x} \int_{|z-x|< \sqrt{t}} \frac{|q(z)|}{|z-x|^{d-2}}\,dz\right]=0\,, \quad \mbox{for}\quad d\geq 3\,, \label{start3}\\
\lim_{t\to 0^+} \left[\sup_{x} \int_{|z-x|< \sqrt{t}} |q(z)|\ln \frac{t}{|z-x|^2}\,dz\right]=0\,, \quad \mbox{for}\quad d= 2\,,  \label{start2}\\
\left[\sup_{x} \int_{|z-x|< 1} |q(z)|\,dz\right]<\infty\,, \quad \mbox{for}\quad d= 1\,. \nonumber 
\end{align}
Here we also refer to Simon \cite[Proposition A.2.6]{MR670130}, Chung and Zhao \cite[Theorem 3.6]{MR1329992}, Demuth and van Casteren \cite[Theorem 1.27]{MR1772266}.
The above remains true if $\ln (t/|z-x|^2)$ is replaced by $\ln(1/|z-x|)$ for $d=2$ and if  $|q(z)|$ is multiplied by $|z-x|$ for $d=1$. In fact, the expressions in square brackets  of \eqref{start} and \eqref{start3} are comparable for $d\geq 3$, while for $d=2$ and $d=1$ similar but slightly different results hold  (see Bogdan and Szczypkowski \cite{MR3200161}, Demuth and van Casteren \cite[Theorem 1.28]{MR1772266}).
We emphasise that \eqref{start3} was used by Kato \cite{MR0203473} to prove by analytic methods that the operator $-\Delta+q$ is essentially self-adjoint (see \cite{MR0333833} for extensions to second order elliptic operators).  The equivalence of \eqref{start} with \eqref{start3} and \eqref{start2} follows also 
from
Theorem~\ref{thm:Gen_1} (see~\cite{MR1132313}). The one-dimensional case is covered by Theorem~\ref{thm:0regH0} of this paper.

In what follows we present and explain our main ideas in view of the literature. 
A~major contribution to the understanding of the subject in a general probabilistic manner is made by Zhao \cite{MR1132313}. Zhao considers a Hunt process $X=(\Omega, \mathscr{F}_t,X_t,\vartheta_t,\PP^x)$ with state space $(S,\rho)$ and life-time $\zeta$, where $S$ is a locally compact metric space with a metric $\rho$ (see \cite{MR0264757}). For a~strong sub-additive functional $A_t$ of $X$, $t\geq 0$, he discusses relations between the following three conditions
\begin{align}
\lim_{r\to 0^+} \left\{\sup_x \EE^x\left[\int_0^{\infty}\ind_{B(x,r)}(X_t)\, dA_t \right] \right\}=0\,, \tag{C1} \label{C1} \\
\lim_{t\to 0^+}  \left[ \sup_x \EE^x (A(t)) \right]=0\,, \tag{C2} \label{C2}\\
\lim_{r\to 0^+} \left\{ \sup_x \EE^x\left[ A(\tau_{B(x,r)}) \right] \right\}=0\,, \tag{C3} \label{C3}
\end{align}
in presence of three hypotheses on the process $X$,
\begin{align}
h_1(X)&\equiv \sup_{t>0} \inf_{r>0} \sup_{x\in S} \,\,\PP^x\left(\tau_{B(x,r)}>t\right)<1\,, \tag{H1} \label{H1}\\
h_2(X)&\equiv \sup_{r>0} \inf_{t>0} \sup_{x\in S}  \,\,\PP^x\left(\tau_{B(x,r)}<t\right)<1\,, \tag{H2} \label{H2}\\
h_3(X)&\equiv \sup_{u>0} \inf_{r>0} \sup_{\substack{x,\,y\in S\\\rho(x,y)\geq u}} \PP^y\left(T_{B(x,r)}<\zeta \right)<1 \,.\tag{H3} \label{H3}
\end{align}
Here for any Borel set $B$ in $S$,
$
T_B
$
is the first hitting time of $B$, $\tau_B=T_{S\setminus B}$ is the first exit time of $B$ (we let $\inf \emptyset=\infty$) and $B(x,r)=\{y\in S: \rho(x,y)<r\}$, $x\in S$, $r>0$.
We present the main theorem of Zhao \cite{MR1132313} on Figure \ref{fig} below; for instance, under (H3), \eqref{C3} implies \eqref{C1}.

\begin{figure}[!hbt]
\begin{center}
\begin{tikzpicture}[->,>=stealth',shorten >=1pt,auto,node distance=3cm,
                    thick,main node/.style={circle,draw,font=\sffamily\Large\bfseries}]

  \node[main node] (1) {C1};
  \node[main node] (2) [right of=1] {C2};
  \node[main node] (3) [right of=2] {C3};

\path
(1) edge [bend left] node {\footnotesize{always}} (3)
(3) edge [bend left] node {\footnotesize{(H3)}} (1)
(2) edge [bend left] node[below] {\footnotesize{(H1)}} (3)
(3) edge [bend left] node[above] {\footnotesize{(H2)}} (2);

\end{tikzpicture}
\caption{Zhao \cite{MR1132313} hypotheses and conditions.}\label{fig}
\end{center}
\end{figure}
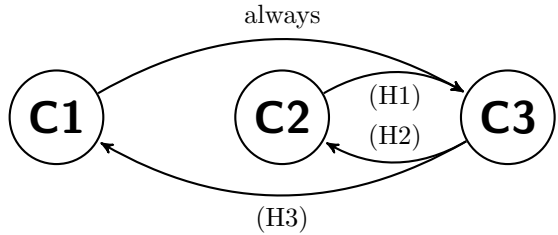
In this paper we assume that $A_t$, $t\geq 0$, is the additive functional of the form
\begin{align}\label{af}
A_t=\int_0^t |q(X_u)|du \,,
\end{align}
and we note that any additive functional is a strong sub-additive functional; see \cite[Lemma 1]{MR1132313}.
Then~\eqref{C2} coincides with \eqref{start} and as such becomes the principal object of our considerations.
We explain the origin and the choice of \eqref{starteqiv}  using the concept of $\lambda$-subprocess $X^{\lambda}$, $\lambda>0$, of~the process $X$ (see \cite{MR0264757} for 
the
definition).
We first notice that \eqref{C2} 
holds for $X$ if and only if
it holds for $X^\lambda$ (see Remark \ref{lem:KK_sub} and Definition~\ref{def:Kato}).
A similar statement is not true in general for \eqref{C1}.
For the standard Brownian motion in $\Rd$, $d\geq 3$, 
\eqref{C2} in fact coincides with \eqref{C1}, which gives rise to \eqref{start3}, yet 
for $d=2$ or $d=1$ the expectation in \eqref{C1}  is infinite for constant non-zero $q$, whereas that never happens for \eqref{C2}.
This shows that \eqref{C1} for $X$ is too strong for a general equivalence result. Therefore
we 
rely on the relations of Figure~\ref{fig} for $X^{\lambda}$,
and then \eqref{C1} results in \eqref{starteqiv}.
We also observe that \eqref{starteqiv} holds for $X$ if and only if it holds for $X^{\lambda '}$, $\lambda '>0$ (see Remark~\ref{lem:KK_sub}).
To ultimately clarify the choice of $X^{\lambda}$ we note that  $h_1(X^\lambda)=h_1(X)$, $h_2(X^\lambda)=h_2(X)$ and $h_3(X^{\lambda})\leq h_3(X)$ (see Lemma~\ref{lem:h1h2_sub} and \ref{lem:h3_sub}).

We now restrict ourselves to the case of the L{\'e}vy process in $\Rd$. Besides being a Hunt process in~$\Rd$, $X$ is also translation invariant. We point out that \eqref{H2} holds for every  L{\'e}vy process and  \eqref{H1} holds if and only if $X$ is not a compound Poisson process (see Remark~\ref{H1H2}).
The case of the compound Poisson process is entirely described in Proposition~\ref{calInKato}.
Thus, in the remaining cases, \eqref{H3} for $X^{\lambda}$ becomes decisive for 
understanding the confines of the
applicability of Figure~\ref{fig} to~$X^{\lambda}$. By Proposition~\ref{lem:h3_sub_Levy} the study of 
$h_3(X^{\lambda})$
 reduces to the analysis of the first hitting time of a single point set by the original L{\'e}vy process~$X$. Namely, we consider 
(see also Lemma~\ref{lem:Bret})
\begin{align}\label{h_lambda}
h^{\lambda}(x)=\EE^0 e^{-\lambda T_{\{x\}}}\,,\qquad x\in\Rd\,.
\end{align}
Eventually, by Corollary~\ref{lem:H3} and Remark~\ref{H1H2}
we obtain the following characterization.
\begin{proposition}\label{prop:summ}
Let $X$ be a L{\'e}vy process in $\Rd$ and $\lambda>0$. All hypotheses \eqref{H1}, \eqref{H2} and \eqref{H3}  are satisfied for
$X^{\lambda}$
if and only if
$\{0\}$ is polar.
\end{proposition}

Therefore
Theorem~\ref{thm:Gen_1} goes much beyond
the range of
\cite{MR1132313}.
The reason is that in our work
we also investigate all the cases 
that are not covered by Figure~1.
Our initial study effects in a list that  
classifies L{\'e}vy processes according to a non-degeneracy hypothesis \Hyp{} 
and specific properties of $h^{\lambda}$,
which is thoroughly examined by Bretagnolle \cite{MR0368175}
for one-dimensional non-Poisson L{\'e}vy processes.
A full layout of our development is presented  in Section~\ref{sec:scheme}.
 Theorem~\ref{thm:Gen_1}  results as a summary of Proposition~\ref{calInKato} and 6 theorems of Section~\ref{sec:thm}.
We stress that the non-symmetric cases or 
those close to 
the compound Poisson process (without \Hyp{}) are more delicate and require more precision.

In \cite[Lemma~4]{MR1132313} Zhao proposes sufficient conditions on $X$ under which \eqref{H1}-\eqref{H3} are satisfied for $X^{\lambda}$.
He uses them to re-prove the result of Aizenman and Simon \cite{MR644024} for $d\geq 2$.
He also verifies hypotheses \eqref{H1}-\eqref{H3} directly for $X$ in the case of L{\'e}vy processes admitting rotationally symmetric transition density with additional assumption on the behaviour of the density 
integrated in time
\cite[Lemma~5]{MR1132313}. Finally he applies that to describe \eqref{start} for symmetric $\alpha$-stable processes, $d>\alpha$, and the relativistic process. We generalize \cite[Lemma~5]{MR1132313} in Theorem~\ref{prop:G_unb_b}.

The paper is organized as follows. In Section~2 we introduce the non-degeneracy hypothesis \Hyp{} for a L{\'e}vy process. Next, we give a classification of L{\'e}vy processes that provides a detailed plan of our research. In the last part of Section~2 we prove 
results 
concerning
hypotheses \eqref{H1}-\eqref{H3}.
In Section~3, for a Hunt process $X$, we define Kato classes $\KK(X)$ and $\pK(X)$ of functions $q$ satisfying \eqref{start} and \eqref{starteqiv}, respectively.
We give other general descriptions of both of those classes and we establish their
initial relations for L{\'e}vy processes. In~Section~4 we prove the main description theorems for L{\'e}vy processes,  separately under and without \Hyp{}. Section~4 ends with additional equivalence results involving the class $\pK^0(X)$ (see \eqref{cond_G_infty^0}).
In Section~5 we present a supplementary discussion on isotropic unimodal L{\'e}vy processes and subordinators. The paper finishes with
examples.

\section{Preliminaries}\label{sec:prel}

Our main focus in this paper is on
a (general) L{\'e}vy process $X$ in $\Rd$ (see \cite{MR1739520}).
The characteristic exponent $\psi$ of $X$ defined by $\EE^0 e^{i\left<x, X_t \right>}=e^{-t \psi(x)}$ equals
\begin{align*}
\psi(x)=-i\left<x,\gamma\right>+\left<x,Ax\right>-\int_{\R^d}\left(e^{i\left<x,z\right>}-1-i\left<x,z\right>\ind_{|z|<1}\right)\nu(dz), \quad x\in\R^d,
\end{align*}
where  $\gamma\in\R^d$, $A$ is  a symmetric non-negative definite matrix and $\nu$ is a L\'{e}vy measure, i.e., $\nu(\{0\})=0$, $\int_{\Rd}\left(1\wedge|z|^2\right)\nu(dz)<\infty$.
If $\int_{\Rd}\left(1\wedge|z|\right)\nu(dz)<\infty$, then the above representation simplifies to
\begin{align*}
\psi(x)=-i\left<x,\gamma_0\right>+\left<x,Ax\right>-\int_{\R^d}\left(e^{i\left<x,z\right>}-1\right)\nu(dz), \quad x\in\R^d,
\end{align*}
where $\gamma_0=\gamma-\int_{\Rd}z\ind_{|z|<1}\nu(dz)$.
Further, if $\gamma_0=0$, $A=0$ and $\nu(\Rd)<\infty$, then $X$ is called a compound Poisson process (see \cite[Remark 27.3]{MR1739520}). We say that $X$ is non-Poisson if $X$ is not a compound Poisson process.
Recall that $\EE^x F(X)=\EE^0 F(X+x)$ for $x\in\Rd$ and Borel
functions $F\geq 0$ on paths. In particular $h^{\lambda}(x)=\EE^{(-x)}e^{-\lambda T_{\{0\}}}$, and thus the following holds.
\begin{remark}\label{rem:0_pol_h}
$\{0\}$ is polar if and only if $h^{\lambda}(x)=0$, $x\in\Rd$.
\end{remark}
\begin{remark}\label{rem:0_reg_h}
$0$ is regular for $\{0\}$ if and only if $h^{\lambda}(0)=1$.
\end{remark}
\begin{remark}\label{rem:var}
$X$ is such that $A=0$, $\gamma_0\in\Rd$, $\int_{\Rd}(|x|\wedge 1)\nu(dx)<\infty$ if and only if $X$ has finite variation on finite time intervals (\cite[Theorem~21.9]{MR1739520}). Then  $\PP^0(\lim_{s\to 0^+} s^{-1}X_s=\gamma_0)=1$ 
(\cite[Theorem~1]{MR0183022}; see also
 \cite[Theorem~43.20]{MR1739520}).
\end{remark}

\begin{lemma}\label{lem:al_sure_0}
Let $X$ be non-Poisson. Then $\PP^0(X_t=0)=0$ except for countably many $t>0$.
\end{lemma}
\begin{proof}
By \cite[Theorem 27.4]{MR1739520} it suffices to consider compound Poisson process with non-zero drift.
Let then $\nu$ and $\gamma_0$ be its L{\'e}vy measure and drift.
According to the decomposition $\nu=\nu^d+\nu^c$ for discrete and continuous part (see \cite[Chapter 5, Section 27]{MR1739520}) we write $X_t=X_t^d+X_t^c+\gamma_0 t$. For $t>0$, by \cite[Remark 27.3]{MR1739520} $\PP^0(X_t^c\in dz)$ is continuous on $\Rd\setminus\{0\}$, therefore $\PP^0(X_t^c\in C\setminus \{0\})=0$ for any countable set $C\subset \Rd$.
By \cite[Corollary 27.5 and Proposition 27.6]{MR1739520} there is a countable set $C_{X^d}\subset \Rd$ such that $\PP^0(X_t^d +\gamma_0 t=0)>0$ if and only if $(-\gamma_0 t) \in C_{X^d}$. Thus $\PP^0(X_t^d+\gamma_0 t=0)=0$ except for countably many $t>0$. Finally,
\begin{align*}
\PP^0(X_t^d+X_t^c+\gamma_0 t=0)
&=\PP^0(X_t^c=0, \,X_t^d+\gamma_0 t=0) +\PP^0(X_t^c=-(X_t^d+\gamma_0 t),\, X_t^d+\gamma_0 t\neq 0)\\
&\leq \PP^0(X_t^d+\gamma_0 t=0) +\PP^0(X_t^c\in -(C_{X^d}+\gamma_0 t)\setminus \{0\})=0\,,
\end{align*}
except for countably many $t>0$.
\end{proof}

We say that a L{\'e}vy process $X$ is non-sticky
if $\PP^0(\tau_{\{0\}}>0)=0$,
or
equivalently 
that
the hypothesis (H) from \cite{MR0368175} holds. 
Lemma~\ref{lem:al_sure_0} reinforces remarks following  \cite[Lemma~3]{MR1132313}.
\begin{remark}\label{rem:non_stick}
$X$ 
is non-sticky  if and only if $X$ is non-Poisson.
\end{remark}

If necessary we specify which L{\'e}vy process we have in mind by adding a superscript, for~instance $h^{Z,\lambda}$ is the function given by \eqref{h_lambda} that corresponds to the process $Z$.

\subsection{Non-degeneracy hypothesis \Hyp{} for L{\'e}vy processes }

Before we introduce the main non-degeneracy hypothesis on a L{\'e}vy process $X$
we recall the basic matrix notation.
We let $M^*$ to be the transpose and $M(\Rd)$ the range of $M$. 
We call $M$ a projection if it is symmetric and $M^2=M$.
For a subset $V$ by $V^{\bot}$ we denote the orthogonal complement of $V$ in $\Rd$. 
We use the following fact.

\begin{lemma}\label{lem:matrix}
If $A$ is symmetric non-negative definite
and $M^* A M=0$, then $A(\Rd)\subseteq M(\Rd)^{\bot}$.
\end{lemma}

\begin{remark}\label{obrot}
Let $X$ be a L{\'e}vy process in a linear subspace $V$ of $\Rd$ (see \cite[Proposition 24.17]{MR1739520}) and denote $d_0=\dim(V)$. Then there exists a rotation given by a matrix $O\in \mathcal{M}_{d\times d}$ such that $Y=OX$ is a L{\'e}vy process in $\R^{d_0}$; the correspondence between $X$ and $Y$ is one-to-one.
\end{remark}
\begin{lemma}\label{polar_rzut}
Let $X$ be a L{\'e}vy process in $\Rd$ and $\mathit{\Pi}$ be a projection. If $\{0\}$ is polar for the process $Y=\mathit{\Pi} X$, then $\{0\}$ is polar for $X$.
\end{lemma}
\begin{proof}
If $X_t+x=0$, then $Y_t +\mathit{\Pi}x=0$, thus
$
\inf\{t>0\colon X_t+x=0\}\geq \inf\{t>0\colon Y_t +\mathit{\Pi}x=0\}
$
and
$
\PP^x(T_{\{0\}}<\infty)\leq \PP^{\mathit{\Pi}x}(T_{\{0\}}^Y <\infty)=0
$.
\end{proof}

\begin{defn}
We say that \Hyp{} holds for $X$ if
there is no linear subspace $V$ of $\Rd$ such that
\begin{align}\label{HypHe}
\dim(V)\leq \min\{1,d-1\}&,\nonumber\\
A(\Rd)\subseteq V,\quad  \nu(\Rd\setminus V)<\infty,
\quad \mathrm{and}\quad \gamma -&\int_{\Rd\setminus V}z\ind_{B(0,1)}(z) \nu(dz) \in V.
\end{align}
\end{defn}

We give a precise probabilistic description of \Hyp{}. 

\begin{remark}\label{opisH0}
For $d=1$, \Hyp{} holds if and only if $X$ is non-Poisson.
For $d> 1$, \Hyp{} holds if and only if $X$ is non-Poisson
and is not of the form \eqref{eq:descr_X} below.
\end{remark}

\begin{proposition}\label{Poi+nor}
Let $d>1$ and $X$ be non-Poisson. Then \rm \Hyp{} does not hold holds if and only if
\begin{align}\label{eq:descr_X}
 X=Y+Z,
\end{align}
and there exist a linear subspace $V$ of $\Rd$, $\dim(V)=1$, such that
\begin{itemize}
\item[i)] $Y$ and $Z$ are independent,
\item[ii)]  $Y$ is either zero or a compound Poisson process with the L{\'e}vy measure vanishing on $V$,
\item[iii)] $Z$ is not a compound Poisson process,
\item[iv)] $Z$ is supported on $V$. 
\end{itemize}
\end{proposition}
\begin{proof}
Since we assume that $X$ is non-Poisson, if \eqref{HypHe} holds and $\dim(V)\leq \min\{1,d-1\}$, then $\dim(V)=1$. We let $Y$ to be a compound Poisson process with the L\'{e}vy measure  $\nu^Y=[\nu]_{\Rd\setminus V}$ and let $Z$ to be a L{\'e}vy process with the L{\'e}vy triplet $(A,\gamma-\int_{\Rd\setminus V}z\ind_{B(0,1)}(z)\nu(dz), [\nu]_V)$, where $[\nu]_{B}$ denotes the measure $\nu$ restricted to a set $B$. By definition $\psi=\psi^Y+\psi^Z$, hence $X=Y+Z$ and
i), ii) and iii) are satisfied.
The property iv) follows from \cite[Proposition 24.17]{MR1739520}.
Conversely, if $X$ is of the form  \eqref{eq:descr_X}, then its L{\'e}vy triplet is given by $A=A^Z$, $\gamma=\gamma^Z + \int_{\Rd \setminus V}z\ind_{B(0,1)}(z)\nu^Y(dz)$ and $\nu=\nu^Y+\nu^Z$. Then \eqref{HypHe} holds since $\nu=\nu^Y$ on $\Rd\setminus V$.
\end{proof}

\noindent
The hypothesis
\Hyp{} agrees with the hypothesis~(H) from~\cite{MR0368175}
if $d=1$.
In particular, for $d=1$ under \Hyp{} we have that $\{0\}$ is essentially polar if and only if $\{0\}$ is polar. As known, in $d>1$ $\{0\}$ is 
always
essentially polar (see \cite[Theorem~16 and Corollary~17]{MR1406564}).

\begin{proposition}\label{polar}
Let $d>1$ and assume {\rm \Hyp{}}. Then $\{0\}$ is polar.
\end{proposition}

\begin{proof}
Let $V$ be the smallest in dimension linear subspace in $\Rd$ satisfying
\eqref{HypHe}.
Now, let $\mathit{\Pi}_1$ be the projection on $V$ and define $Y=\mathit{\Pi}_1X$. Observe that by \Hyp{} we have $\dim(V)\geq 2$.
We claim that there is no one-dimensional subspace $W\subset V$ such that the projection of $Y$ on $W$ is a compound Poisson process. For the proof assume that there is such $W$ and let $\mathit{\Pi}_2$ be the projection on $W$. Then $Z=\mathit{\Pi}_2Y=\mathit{\Pi}_2X$ is a compound Poisson process.  By \cite[Proposition~11.10]{MR1739520} we have the following consequences. First, $\mathit{\Pi}_2A\mathit{\Pi}_2=0$ and by Lemma \ref{lem:matrix} we obtain $A(\Rd)\subseteq V \cap W^{\bot}$. Next, $\nu(\Rd\setminus W^{\bot})=\nu \mathit{\Pi}_2^{-1}(\Rd\setminus \{0\})<\infty$ and then $\nu(\Rd\setminus (V\cap W^{\bot}))<\infty$. Further,
since $\mathit{\Pi}_2z=0$ on $V\cap W^{\bot}$ we have
\begin{align*}
0&= \mathit{\Pi}_2\gamma - \int_{\Rd}\mathit{\Pi}_2z\ind_{B(0,1)}(z)\nu(dz)\\
&=\mathit{\Pi}_2\gamma - \int_{\Rd\setminus(V\cap W^{\bot})} \mathit{\Pi}_2z \ind_{B(0,1)}(z)\nu(dz)\\
&=\mathit{\Pi}_2\left(\gamma -\int_{\Rd\setminus(V\cap W^\bot)}z\ind_{B(0,1)}(z)\nu(dz)\right)\,.
\end{align*}
Thus $\gamma_1=\gamma -\int_{\Rd\setminus(V\cap W^\bot)}z\ind_{B(0,1)(z)}\nu(dz) \in W^{\bot}$. Finally, by $\Rd\setminus (V\cap W^\bot)=(\Rd\setminus V )\dot\cup (V\setminus W^\bot)$ and by \eqref{HypHe},
\begin{align*}
\gamma_1=\left(\gamma -\int_{\Rd\setminus V}z\ind_{B(0,1)}(z)\nu(dz)\right)-\int_{V\setminus W^\bot}z\ind_{B(0,1)}(z)\nu(dz)\in V\,,
\end{align*}
which is a contradiction, because then \eqref{HypHe} holds with $V\cap W^{\bot}$ in place of $V$ and $\dim(V\cap W^{\bot})<\dim(V)$.
Now, by Remark \ref{obrot} we can treat $Y$ as a process in $\R^{d_0}$, $d_0=\dim(V) \geq 2$, and then by   \cite[Theoreme 4]{MR0368175} the set $\{0\}$ is a polar set for $Y$ as well as for $X$ by Lemma~\ref{polar_rzut}.
\end{proof}

\subsection{Classification of L{\'e}vy processes}\label{sec:scheme}

We outline our work-flow to analyze every L{\'e}vy process $X$.
\noindent
Exclusively one of the following situations holds for a L{\'e}vy process in $\Rd$.
\begin{enumerate}
\item \Hyp{} holds:
\begin{enumerate}
\item $d>1$ 
(
then 
$h^{\lambda}(x)=0$, $x\in\Rd$),
\item $d=1$
\begin{enumerate}
\item[]
\begin{enumerate}
\item[(A)] $h^{\lambda}(x)=0$, $x\in\R$,
\item[(B)] $h^{\lambda}(0)=\liminf_{x\to 0} h^{\lambda}(x)<\limsup_{x\to 0} h^{\lambda}(x)=1$,
\item[(C)] $h^{\lambda}(0)=\lim_{x\to 0} h^{\lambda}(x)=1$.
\end{enumerate}
\end{enumerate}
\end{enumerate}
\item \Hyp{} does not hold:
\begin{enumerate}
\item a compound Poissson process ($d\geq 1$; then $h^{\lambda}(0)=1$),
\item given by \eqref{eq:descr_X} ($d>1$)
\begin{enumerate}
\item[]
\begin{enumerate}
\item[(A')] $h^{Z,\lambda}(v)=0$, $v\in V$,
\item[(B')] $h^{Z,\lambda}(0)=\liminf_{v\in V,\,v\to 0} h^{Z,\lambda}(v)<\limsup_{v\in V,\, v\to 0} h^{Z,\lambda}(v)=1$,
\item[(C')] $h^{Z,\lambda}(0)=\lim_{v\in V,\, v\to 0} h^{Z,\lambda}(v)=1$.
\end{enumerate}
\end{enumerate}
\end{enumerate}
\end{enumerate}
The comment in the case case 1(a)
is a consequence of 
Proposition~\ref{polar} and Remark~\ref{rem:0_pol_h}.
The partition of the case 1(b) is due to
Remark~\ref{opisH0}, \ref{rem:non_stick} and \cite[Theoreme~3 and~6]{MR0368175}.
The division of the case 2 results from
Remark~\ref{opisH0}. 
The subcases of 2(b) follow from  Remark~\ref{obrot} 
and \cite{MR0368175}.

The subcases of 1(b) translate equivalently into probabilistic properties of $X$, see \cite[Theoreme 6, 8]{MR0368175} and Remark~\ref{rem:var}. We have
\begin{enumerate}
\item[]
\begin{enumerate}
\item[]
\begin{enumerate}
\item[]
\begin{enumerate}
\item[(A)] $\{0\}$ is polar,
\item[(B)] X has finite variation and non-zero drift,
\item[(C)] $0$ is regular for $\{0\}$.
\end{enumerate}
\end{enumerate}
\end{enumerate}
\end{enumerate}
The analytic counterpart by means of the characteristic exponent or the L{\'e}vy triplet is (see \cite[Theoreme~3, 7 and 8]{MR0368175})
\begin{enumerate}
\item[]
\begin{enumerate}
\item[]
\begin{enumerate}
\item[]
\begin{enumerate}
\item[(A)] $\int_{\R}\mathrm{Re}\left(\frac{1}{\lambda+\psi(z)}\right)dz=\infty$,
\item[(B)] $A=0$, $\gamma_0\neq 0$  and $\int_{\R}(|x|\wedge 1) \nu(dx)<\infty$,
\item[(C)] $A\neq 0$ or (A) does not hold and $\int_{\R}(|x|\wedge 1) \nu(dx)=\infty$.
\end{enumerate}
\end{enumerate}
\end{enumerate}
\end{enumerate}

We could similarly reformulate 2(b) for $Z$,
but in proofs of Theorem~\ref{prop:WH0_bouVar1} and~\ref{thm:WH0_0reg} we use the following description.
\begin{enumerate}
\item[]
\begin{enumerate}
\item[]
\begin{enumerate}
\item[]
\begin{enumerate}
\item[(A')] $\int_{V}\mathrm{Re}\left(\frac{1}{\lambda+\psi^Z(v)}\right)dv=\infty$ ($dv$ is the one-dimensional Lebesgue measure~on~$V$),
\item[(B')] $A^Z=0$, $\gamma_0^Z\neq 0$  and $\int_V(|x|\wedge 1) \nu^Z(dx)<\infty$,
\item[(C')] $0$ is regular for $\{0\}$.
\end{enumerate}
\end{enumerate}
\end{enumerate}
\end{enumerate}

We translate (A'), (B') and (C') into $X$ given by \eqref{eq:descr_X}.
\begin{lemma}\label{lem:pZ}
$\{0\}$ is polar for $X$ if and only if $\{0\}$ is polar for $Z$.
\end{lemma}
\begin{proof}
If $\{0\}$ is polar for $Z$, then 
$\int_{V}\mathrm{Re}(1/[\lambda+\psi^Z(v)])dv=\infty$. By Lemma~\ref{polar_rzut} to verify that $\{0\}$ is polar for $X$
it suffices to show that it is polar for $\mathit{\Pi}X=\mathit{\Pi}(Y+Z)=\mathit{\Pi}Y+Z$, where $\mathit{\Pi}$ is the projection on $V$. Since $\psi^{\mathit{\Pi} X}=\psi^{\mathit{\Pi} Y}+\psi^Z$ and $\psi^{\mathit{\Pi} Y}$ is bounded ($\mathit{\Pi} Y$ is a compound Poisson process) we have by our assumption
$\int_V \mathrm{Re}(1/[\lambda +\psi^{\mathit{\Pi} X}(v)])dv=\infty$.
Thus Remark~\ref{obrot} and \cite[Theoreme 7, 3]{MR0368175} end this part of the proof.
If $\{0\}$ is not polar for $Z$, $\PP^0(T_{\{x\}}^Z<\infty)>0$ for some $x\in V$, we have for large $t>0$
$$
\PP^0 (T_{\{x\}}<\infty)\geq \PP^0 (Y_t=0,\, T_{\{x\}}=T_{\{x\}}^Z<t) = \PP^0 (Y_t=0)\PP^0( T_{\{x\}}^Z<t)>0\,.
$$
\end{proof}

\begin{lemma}\label{lem:Levy_hlambda}
$\{0\}$ is not polar for $X$ if and only if $\limsup_{x\to 0} h^{\lambda}(x)=1$.
\end{lemma}

\begin{proof}
If $\limsup_{x\to 0} h^{\lambda}(x)=1$, then $h^{\lambda}(x)>0$ for some $x\in\Rd$ and $\PP^0(T_{\{x\}}<\infty)>0$. Conversely, if $\{0\}$ is not polar for $X$ then by Lemma \ref{lem:pZ} it is not polar for $Z$ and $\limsup_{v\in V, v\to 0} h^{Z,\lambda}(v)=1$. This implies $\limsup_{v\in V, v\to 0} \PP^0(T_{\{v\}}^Z<t)=1$ for every fixed $t>0$. Thus we have for $t>0$
\begin{align*}
h^{\lambda}(x)
&\geq \EE^0\left( Y_t=0,T_{\{x\}}^Z<t;e^{-\lambda T_{\{x\}}}\right)
= \EE^0\left( Y_t=0, T_{\{x\}}^Z<t;e^{-\lambda T_{\{x\}}^Z}\right)\\
&\geq \PP^0(Y_t=0)\PP^0(T_{\{x\}}^Z<t) e^{-\lambda t}\,,
\end{align*}
which gives $\limsup_{x\to 0} h^{\lambda}(x)\geq \PP^0(Y_t=0) e^{-\lambda t}$.
Finally, we let $t\to 0^+$.
\end{proof}

\begin{lemma}
$0$ is regular for $\{0\}$ for $X$ if and only if
$0$ is regular for $\{0\}$ for
$Z$.
\end{lemma}
\begin{proof}
We observe that
the set $\{Y_s=0 \mbox{ for all } s\in [0,\delta] \mbox{ for some } \delta>0\}$ is of measure one with respect to $\PP^0$.
On that set $T_{\{0\}}=0$ if and only if $T_{\{0\}}^Z=0$.
\end{proof}

\begin{corollary}\label{lem:naX}
For the process $X$ of the form \eqref{eq:descr_X}  we have
\begin{enumerate}
\item[]
\begin{enumerate}
\item[{\rm (A')}] $h^{\lambda}(x)=0$, $x\in\Rd$,
\item[{\rm (B')}] $h^{\lambda}(0)<\limsup_{x\to 0}h^{\lambda}(x)=1$,
\item[{\rm (C')}] $h^{\lambda}(0)=\limsup_{x\to 0}h^{\lambda}(x)=1$,
\end{enumerate}
\end{enumerate}
and
\begin{enumerate}
\item[]
\begin{enumerate}
\item[{\rm (A')}] $\{0\}$ is polar,
\item[{\rm (B')}] X has finite variation and non-zero drift (see Remark \ref{rem:var}),
\item[{\rm (C')}] $0$ is regular for $\{0\}$.
\end{enumerate}
\end{enumerate}
\end{corollary}

\noindent
The last observation facilitates a discussion of \eqref{H3} in the next subsection.
\begin{remark}\label{rem:Levy_hlambda}
For a non-Poisson L{\'e}vy process we have $\limsup_{x\to 0} h^{\lambda}(x)=1$ or $h^{\lambda}(x)=0$, $x\in\Rd$.
\end{remark}

\subsection{Hypotheses (H1)-(H3)}

We start with a general case of a Hunt process $X$ on $S$ with life-time  $\zeta$.
In the proofs of Lemma~\ref{lem:h1h2_sub} and \ref{lem:h3_sub} all objects corresponding to $X^{\lambda}$, the $\lambda$-subprocess of $X$, are indicated with a bar, e.g., $\overline{T}_B=\inf \{ t>0 \colon X_t^{\lambda}\in B \}$.
\begin{lemma}\label{lem:h1h2_sub}
Let $\lambda>0$. We have $h_1(X^\lambda)=h_1(X)$ and $h_2(X^\lambda)=h_2(X)$.
\end{lemma}
\begin{proof}
Recall that $\inf \emptyset =\infty$. For any Borel set $B$ in $S$ and $t>0$ we have
$\{\overline{\tau}_B>t\}=
{\{\tau_B>t\}\times [0,\infty)}\, \dot\cup\, \{\tau_B\leq t\}\times [0,\tau_B)$
and $\{\overline{\tau}_B<t\}=\{\tau_B<t\}\times (\tau_B,\infty)$. Thus,
\begin{align*}
\overline{\PP}^x\left(\overline{\tau}_B >t\right)&
= \PP^x \left(\tau_B>t \right)+ \EE^x \left(\tau_B\leq t; 1-e^{-\lambda \tau_B}  \right)
\leq \PP^x \left(\tau_B>t \right)+ 1-e^{-\lambda t} \,,
\end{align*}
and
\begin{align*}
\overline{\PP}^x\left(\overline{\tau}_B <t\right)&=\EE^x\left(\tau_B<t;\, e^{-\lambda \tau_{B}} \right)
= \PP^x \left(\tau_B<t \right)+ \EE^x \left(\tau_B< t;\,e^{-\lambda \tau_B}-1 \right)\\
&\geq \PP^x \left(\tau_B<t \right)+e^{-\lambda t}-1\,.
\end{align*}
Since we may change $\sup_{t>0}$ with $\limsup_{t\to 0^+}$, $h_1(X)\leq h_1(X^{\lambda})\leq h_1(X)+\lim_{t\to 0^+}(1-e^{-\lambda t})$ and since we may replace $\inf_{t>0}$ with $\liminf_{t\to 0^+}$, $h_2(X)\geq h_2(X^{\lambda})\geq h_2(X)+\lim_{t\to 0^+}(e^{-\lambda t}-1)$. This ends the proof.
\end{proof}

\begin{lemma}\label{lem:h3_sub}
Let $\lambda>0$. We have $h_3(X^{\lambda})\leq h_3(X)$, more precisely
$$
h_3(X^{\lambda})=\sup_{u>0} \inf_{r>0} \sup_{\substack{x,\,y\in S\\\rho(x,y)\geq u}} \EE^y (T_{B(x,r)}<\zeta;e^{-\lambda T_{B(x,r)}})\,.
$$
\end{lemma}
\begin{proof}
For any Borel set $B$ in $S$ we have $\{\overline{T}_{B}<\overline{\zeta}\}=\{T_B<\zeta\}\times (T_B,\infty)$. This results in $\overline{\PP}^y(\overline{T}_{B}<\overline{\zeta})=\EE^y(T_B<\zeta;e^{-\lambda T_B})$.
\end{proof}

Now, let $S=\Rd$ be the Euclidean space and $\zeta=\infty$.
The following lemmas and corollary address the question whether $h_3(X^{\lambda})=\sup_{u>0} \inf_{r>0} \sup\limits_{|x-y|\geq u} \EE^y e^{-\lambda T_{B(x,r)}}<1$.

\begin{lemma}\label{lem:point}
Let $x\in\Rd$ be fixed. Then
\begin{align}\label{eq:point}
\lim_{r\to 0^+} T_{\overline{B}(x,r)}=T_{\{x\}}\qquad \PP^0 a.s.
\end{align}
\end{lemma}
\begin{proof}
Fix $x\in\Rd$. Define the stopping times $T_r=T_{\overline{B}(x,r)}$ and $T=\lim_{r\to 0^+}T_r$, $r>0$. Obviously, $T_r\leq T\leq T_{\{x\}}$. It suffices to consider \eqref{eq:point} on the set $\{T<\infty\}$, otherwise both sides of \eqref{eq:point} are infinite.
Since $T_r$ is non-increasing in $r>0$ we have  by the quasi-left continuity $\lim_{r\to 0^+} X_{T_r}=X_{T}$\, a.s. on $\{T<\infty\}$.  
On the other hand, by the right continuity we have $X_{T_r}\in \overline{B}(x,r)$ and thus $\lim_{r\to 0^+}X_{T_r}=x$ a.s. on $\{T<\infty\}$.
Finally, $X_T=x$ and consequently $T\geq T_{\{x\}}$ a.s. on $\{T<\infty\}$.
\end{proof}

\begin{lemma}\label{lem:tau}
Let $\tau_n=\tau_{B(0,n)}$. Then $\lim_{n\to\infty}\tau_n=\infty$ $\PP^0$  a.s.
\end{lemma}
\begin{proof}
Denote $\tau=\lim_{n\to\infty}\tau_n$. Since $\tau_n$ is non-decreasing, by the quasi-left continuity $X_{\tau_n} \xrightarrow {n\to \infty} X_{\tau}$ a.s. on $\{\tau<\infty\}$. On $\{\tau<\infty\}$ for $n\geq |X_{\tau}|+1$ by the right continuity we have $|X_{\tau_n}|\geq |X_\tau|+1$, which is a contradiction; it shows that a.s $\tau<\infty$ does not occur.
\end{proof}

\begin{lemma}\label{lem:ZhH3}
Let $\lambda>0$.
Then
\begin{align}\label{ZhaoH3}
\sup_{u>0}\inf_{r>0}\sup_{|x|\geq u}\EE^0e^{-\lambda T_{\overline{B}(x,r)}}=\sup_{x\neq 0}\EE^0e^{-\lambda T_{\{x\}}}\,.
\end{align}
\end{lemma}
\begin{proof}
Let $f_r(x)=\EE^0e^{-\lambda T_{\overline{B}(x,r)}}$, $r\geq 0$, $x\in \Rd$, where $\overline{B}(x,0)=\{x\}$. Notice that
$f_r (x)\geq f_0(x)$. Therefore
\begin{align}\label{ineq:fr1}
a=\sup_{u>0}\inf_{r>0} \sup_{|x|\geq u} f_r (x)\geq \sup_{u>0}\inf_{r>0} \sup_{|x|\geq u} f_0(x)= \sup_{u>0,\,|x|\geq u} f_0(x) = \sup_{x\neq 0}f_0(x)\geq 0\,.
\end{align}
It suffices to prove the reverse inequality in the case $a\neq 0$, otherwise \eqref{ZhaoH3} holds by \eqref{ineq:fr1}. Thus let $a\in (0,1]$. Then for $\varepsilon>0$ there is $u>0$ such that for all $r>0$ we have $\sup_{|x|\geq u}f_r(x) >a- \varepsilon$.
Hence, there is a sequence $\{x_n\}$ such that $f_{1/n}(x_n)> a-\varepsilon$ and $|x_n|\geq u$.
We will show that $\{x_n\}$ is bounded. For $r\in(0,1]$, $m\in\NN$ and $|x|\geq m+2$, we have $T_{\overline{B}(x,r)}\geq \tau_{B_m}$ thus by Lemma~\ref{lem:tau} and the dominated convergence theorem there is $m_0$ such that
\begin{align*}
\sup_{|x|\geq m_0+2} f_r(x)\leq \EE^0 e^{-\lambda \tau_{m_0}}\leq a-\varepsilon\,.
\end{align*}
This proves that $m_0+2\geq |x_n|\geq u>0$ for every $n$. We let $y\neq 0$ to be the limit point of $\{x_n\}$. Observe that for every $r>0$ there is $n$  such that $B(x_n,1/n)\subseteq B(y,r)$, which implies $T_{\overline{B}(y,r)}\leq T_{\overline{B}(x_n,1/n)}$ and $f_r(y)\geq f_{1/n}(x_n)>a-\varepsilon$. Finally, by Lemma \ref{lem:point} and the dominated convergence theorem we obtain
\begin{eqnarray*}
\sup_{x\neq 0}\EE^0e^{-\lambda T_{\{x\}}}\geq \EE^0e^{-\lambda T_{\{y\}}} =\lim_{r\to 0}\EE^0e^{-\lambda T_{\overline{B}(y,r)}} = \lim_{r\to 0}f_r(y)\geq a-\varepsilon.
\end{eqnarray*}
This ends the proof since $\varepsilon>0$ was arbitrary.
\end{proof}

We continue discussing \eqref{H1}-\eqref{H3} for a L{\'e}vy process $X$ in $\Rd$.
Remark~\ref{rem:non_stick} and \cite[Lemma~2 and~3]{MR1132313} 
ensure the following.

\begin{remark}\label{H1H2}
Clearly \eqref{H1} does not hold for any compound Poisson process.\\
\eqref{H1} holds for every non-Poisson L{\'e}vy process $X$ with $h_1(X)=0$.\\
\eqref{H2} holds for every L{\'e}vy process $X$ with $h_2(X)=0$.\\
\end{remark}

\begin{proposition}\label{lem:h3_sub_Levy}
Let $X$ be a L{\'e}vy process in $\Rd$ and $\lambda>0$. For $h^{\lambda}$ defined in \eqref{h_lambda} we have
$$
h_3(X^{\lambda})=\sup_{x\neq 0} h^{\lambda}(x)\,.
$$
\end{proposition}
\begin{proof}
By Lemma \ref{lem:h3_sub}, $\overline{B}(x,r/2)\subseteq B(x,r)\subseteq \overline{B}(x,r)$ and Lemma \ref{lem:ZhH3}
\begin{align*}
h_3(X^{\lambda})&=\sup_{u>0} \inf_{r>0} \sup_{|x-y|\geq u} \EE^y (T_{B(x,r)}<\infty;e^{-\lambda T_{B(x,r)}})
=\sup_{u>0} \inf_{r>0} \sup_{|x-y|\geq u} \EE^0 (e^{-\lambda T_{\overline{B}(x-y,r)}})\\
&= \sup_{x\neq 0} \EE^0 e^{-\lambda T_{\{x\}}}\,.
\end{align*}
\end{proof}

\noindent
By Proposition~\ref{lem:h3_sub_Levy}, Remark~\ref{rem:Levy_hlambda} 
and~\ref{rem:0_pol_h} we obtain an improvement of \cite[Lemma 4]{MR1132313}.

\begin{corollary}\label{lem:H3}
Let $X$ be non-Poisson and $\lambda>0$. Then \eqref{H3} holds for $X^{\lambda}$ if and only if
\mbox{$\{0\}$ is polar} for $X$.
If this is the case, then we have $h_3(X^{\lambda})=0$.
\end{corollary}

\section{Kato class}

Let $X$ be a Hunt process in $\Rd$.
For $t\geq 0$ we define the {\it transition kernel} $P_t(x,dz)$ and
the corresponding {\it transition operator} $P_t$ by
$$
P_t (x,B) =\PP^x (X_t\in B)\,,\qquad
P_t f(x) =\int_{\Rd} f(z)P_t(x,dz)\,.
$$
Moreover, for $\lambda\geq 0$ and $t\in(0,\infty]$ we let
\begin{align*}
G_t^{\lambda}(x,B)=\int_0^t e^{-\lambda s} P_u(x,B)du\,,\qquad
G_t^{\lambda} f(x)= \int_{\Rd} f(z) G_t^{\lambda}(x,dz)
=
\int_0^t e^{-\lambda u} P_u f(x)du
\,,
\end{align*}
to
be
the
(truncated)
{\it $\lambda$-potential kernel} 
and the 
(truncated)
{\it $\lambda$-potential operator} $G_t^{\lambda}$,
respectively.
We simplify the notation by putting $G^{\lambda}(x,dz)=G_{\infty}^{\lambda}(x,dz)$
and $G^{\lambda}=G_{\infty}^{\lambda}$.

\begin{defn}\label{def:Kato}
Let $q:\Rd\to\R$.
We write $q\in \KK(X)$ if \eqref{start} holds, i.e.,
\begin{align}\label{def:KK}
\lim_{t\to 0^+}\left[\sup_{x\in\Rd} G^0_t |q|(x)\right]=0.
\end{align}
We write $q\in\pK(X)$ if  \eqref{starteqiv}  holds for some (every) $\lambda>0$, i.e.,
\begin{align}\label{def:pK}
\lim_{r\to 0^+}\left[\sup_{x\in\Rd}\int_{B(x,r)}|q(z)|\,G^\lambda(x,dz)\right]=0.
\end{align}
\end{defn}

If the process $X$ is understood from the context we will write in short $\KK$, $\pK$ for $\KK(X)$,~$\pK(X)$.
In the next two lemmas we show that the definition of $\pK$ is consistent. The first one is an apparent reinforcement of \eqref{starteqiv} and \eqref{def:pK}.
\begin{lemma}\label{sup_R2}
For all $\lambda\geq 0$, $t\in(0,\infty]$,
\begin{align*}
\left[\sup_{x,y\in\Rd}\int_{B(x,r)}|q(z)|\,G_t^{\lambda}(y,dz)\right]
\leq \left[ \sup_{x\in\Rd} \int_{B(x,2r)}|q(z)|\,G_t^{\lambda}(x,dz)\right]\,,\quad r>0\,.
\end{align*}
\end{lemma}
\begin{proof}
Let $T=T_{\overline{B}(x,r)}$. The strong Markov property leads to
\begin{align*}
&\EE^{y}\left( \int_0^{\infty} e^{-\lambda s} \ind_{(0,t](s)} \ind_{B(x,r)}(X_s) |q(X_s)|\,ds\right)
=\EE^{y}\left(T<\infty; \int_T^{\infty} e^{-\lambda s} \ind_{(0,t]}(s) \ind_{B(x,r)}(X_s) |q(X_s)|\,ds\right)\\
&\leq\EE^{y}\left(T<\infty; e^{-\lambda T}\int_0^{\infty}  e^{-\lambda u} \ind_{(0,t]}(u) \ind_{B(x,r)}(X_u\theta_T) |q(X_u\theta_T)\,du\right)\\
&=\EE^{y}\left(T<\infty; e^{-\lambda T}\EE^{X_T}\left(\int_0^{\infty} e^{-\lambda u} \ind_{(0,t]}(u) \ind_{B(x,r)}(X_u) |q(X_u)|\,du\right)\right)\,,
\end{align*}
where $\theta$ denotes the usual shift operator.
By the right continuity $X_T\in \overline{B}(x,r)$ and $B(x,r)\subseteq B(X_T,2r)$ on $\{T<\infty\}$. Thus eventually
\begin{align*}
\int_{B(x,r)}|q(z)|\,G_{t}^{\lambda}(y,dz)
\leq \EE^{y}\left(T<\infty; e^{-\lambda T}\EE^{X_T}\left(\int_0^{\infty} e^{-\lambda u} \ind_{(0,t]}(u) \ind_{B(X_T,2r)}(X_u) |q(X_u)|\,du\right)\right)\\
\leq \sup_{x\in\Rd}\EE^x\left[ \int_0^{\infty}e^{-\lambda u} \ind_{(0,t](u)} \ind_{B(x,2r)}(X_u) |q(X_u)|\,du\right]
=\sup_{x\in\Rd} \int_{B(x,2r)}|q(z)|\,G_t^{\lambda_0}(x,dz)\,.
\end{align*}
\end{proof}

\begin{lemma}\label{lem:consistency}
If \eqref{starteqiv} or \eqref{def:pK} holds for some $\lambda_0>0$, then it holds for every $\lambda>0$.
\end{lemma}
\begin{proof}
Clearly, by the resolvent formula (see \cite[Chapter 1,  (8.10)]{MR0264757}) it suffices to consider the  measure $A\mapsto \int \ind_A(z)\, G^{\lambda_0}G^{\lambda}(x,dz)=\iint \ind_A(z)G^{\lambda_0}(y,dz)G^{\lambda}(x,dy)$. We have
\begin{align*}
\int_{B(x,r)} |q(z)|\, G^{\lambda_0}G^{\lambda}(x,dz)&=\int_{\Rd} \left(\int_{B(x,r)} |q(z)|\, G^{\lambda_0}(y,dz)\right)G^{\lambda}(x,dy)\\
&\leq \lambda^{-1} \left[\sup_{x,y\in\Rd}\int_{B(x,r)}|q(z)|\,G^{\lambda_0}(y,dz)\right]\,.
\end{align*}
This ends the proof due to Lemma \ref{sup_R2}.
\end{proof}

Now, we give alternative characterisations of $\KK(X)$ and $\pK(X)$.
We easily observe that 
\begin{align}\label{simple}
e^{-\lambda t}\, G_t^0(x,dz)\leq G_t^{\lambda}(x,dz)\leq G_t^0(x,dz)\,.
\end{align}

\begin{lemma}\label{lem:ineq_eqiv}
For $\lambda>0$ and $t\in [1/\lambda,\infty]$ we have
\begin{align*}
(1-e^{-1})\sup_x \big[G^{\lambda}_t |q|(x)\big] \leq \sup_x \left[ G^0_{1/\lambda}|q|(x)
\right] \leq e \sup_x \big[G^{\lambda}_t |q|(x) \big].
\end{align*}
\end{lemma}
\begin{proof}
Actually, the upper bound holds pointwise as follows,
$$G_{1/\lambda}^0|q|(x)=\int_0^{1/\lambda} P_u|q|(x)du\leq e\int_0^{1/\lambda} e^{-\lambda u}P_u |q|(x)du\leq e\,G^{\lambda}_t|q| (x).$$
We prove the lower bound,
\begin{align*}
G^{\lambda}|q|(x)&\leq\sum^\infty_{k=0}e^{-k}\int^{(k+1)/\lambda}_{k/\lambda}P_{k/\lambda}P_{u-k/\lambda}|q|(x)du= \sum^\infty_{k=0}e^{-k}P_{k/\lambda}\(\int^{1/\lambda}_0 P_u|q|(\cdot)du\)(x)\\&\leq (1-e^{-1})^{-1}\sup_{z\in\Rd}\left[\int^{1/\lambda}_0 P_u|q|(z)du\right].
\end{align*}
\end{proof}

Here is a conclusion from \eqref{simple} and Lemma~\ref{lem:ineq_eqiv}.
\begin{proposition}\label{gen_KK_equiv}
The following conditions are equivalent to $q\in \KK(X)$.
\begin{enumerate}
\item[i)] $\lim_{t\to 0^+}\Big[\sup_{x\in \Rd}G^\lambda_{t} |q|(x)\Big]=0$ for some (every) $\lambda \geq 0$.
\item[ii)] $\lim_{\lambda\to \infty}\Big[\sup_{x\in \Rd}G^\lambda_{t} |q|(x)\Big]=0$ for some (every) $t\in (0,\infty]$.
\end{enumerate}
\end{proposition}

For resolvent operators $R^{\lambda}$, $\lambda>0$, of a strongly continuous  contraction semigroup on a Banach space we have $\lim_{\lambda \to \infty}\lambda R^{\lambda}\phi =  \phi$. Thus $\lim_{\lambda\to\infty} R^{\lambda}\phi=0$ in the norm for every element $\phi$ of the Banach space. For a Markov process the counterparts of the resolvent operators are the $\lambda$-potential operators $G^{\lambda}_{\infty}$.

Proposition \ref{gen_KK_equiv} extends the equivalence of (i) and (ii) of  \cite[Theorem~III.1]{MR1054115} from a subclass of L{\'e}vy processes to any Hunt process.
Similar result is proved in \cite[Lemma~3.1]{MR2277833} where authors discuss the Kato class of measures for Markov processes possessing transition densities that satisfy the Nash type estimate (see\
\cite{MR2345907} for the symmetric case).
In Lemma \ref{lem:L_unif} we also show that the 
uniform local integrability of $V$ (\cite[Theorem III.1]{MR1054115}) is 
necessary 
for $V\in \KK(X)$ for any L{\'e}vy process $X$ in $\Rd$.

We briefly explain  the role of Proposition~\ref{gen_KK_equiv}.
For the Brownian motion, as mentioned in \cite{MR1642818}, by Stein's interpolation theorem the inequality $\sup_{x\in\Rd} [G^{\lambda}|q|(x)]\leq \gamma$ leads  to $|\!| |q|^{1/2} \phi |\!|_2^2\leq \gamma\, |\!|\nabla \phi |\!|_2^2 + \lambda |\!|\phi |\!|_2^2 $, $\phi\in C_c^{\infty}(\Rd)$
(a partial reverse result is proved in \cite[Theorem 4.9]{MR644024}).
For a counterpart of such implication  for other processes see remarks preceding \cite[Theorem 4.10]{MR2944475}.
The latter inequality with
$\gamma<1$ allows to define a self-adjoint Schr{\"o}dinger operator in the sense of quadratic forms, cf. \cite[Theorem 3.17]{MR2848339}, the analogue of Kato-Rellich theorem.

We use Lemma~\ref{sup_R2} to get a better insight into the result of Lemma~\ref{lem:ineq_eqiv}.
\begin{lemma}\label{lem:ineq_equiv_impr}
For $t\in(0,\infty)$ we have $G_t^0(x,dz)\leq e\, G^{1/t}(x,dz)$ and
\begin{align*}
(1-e^{-1}) \sup_{x\in\Rd} \left[\int_{B(x,r)}|q(z)| G^{1/t}(x,dz)\right]
\leq \sup_{x\in\Rd} \left[ \int_{B(x,2r)} |q(z)|G_t^0(x,dz)\right],\quad r>0\,.
\end{align*}
\end{lemma}
\begin{proof}
For a fixed $y\in\Rd$ by Lemma \ref{lem:ineq_eqiv} with $\tilde{q}(z)=q(z)\ind_{B(y,r)}(z)$ we have
\begin{align*}
(1-e^{-1}) \int_{B(y,r)}|q(z)| G^{1/t}(y,dz)
&= (1-e^{-1}) G^{1/t}|\tilde{q}|(y)\\
\leq \sup_{x\in\Rd}\int^t_{0} P_s|\tilde{q}|(x)ds
&=\sup_{x\in\Rd} \int_{\Rd} |\tilde{q}(z)|G_t^0(x,dz)
=\sup_{x\in\Rd}\int_{B(y,r)}|q(z)|G_t^0(x,dz).
\end{align*}
Thus, by Lemma \ref{sup_R2} we obtain
\begin{align*}
(1-e^{-1}) \sup_{y\in\Rd} \int_{B(y,r)}|q(z)| G^{1/t}(y,dz) \leq \sup_{x\in\Rd} \int_{B(x,2r)} |q(z)|G_t^0(x,dz)\,.
\end{align*}
\end{proof}

The following is the aftermath of \eqref{simple} and Lemma~\ref{lem:ineq_equiv_impr}.
\begin{proposition}\label{pK_alternat}
$q\in\pK(X)$ if and only if
\begin{align*}
\lim_{r\to 0^+} \left[ \sup_{x\in\Rd}\int_{B(x,r)}|q(z)|G_t^{\lambda}(x,dz)\right]=0\,,
\end{align*}
for some (all) $t\in(0,\infty)$, $\lambda\geq 0$.
\end{proposition}

The above truncation in time is useful when the distribution $\PP^x(X_s\in dz)$ is well estimated only for $s\in(0,t]$ near every $x\in\Rd$. See
\cite{2014arXiv1403.0912K}, \cite[Theorem~2.4 and~3.1]{MR2524930} for such estimates. In view of \cite[(A2.3), Lemma~4.1 and~4.3]{MR2345907}  Proposition~\ref{pK_alternat} can also be regarded as an extension or counterpart of \cite[Theorem~3.1]{MR2345907}.
We use Proposition \ref{pK_alternat} in  Example \ref{example:1} below.

\begin{remark}\label{lem:KK_sub}
Let $\lambda>0$. Then $\KK(X)=\KK(X^{\lambda})$ and $\pK(X)=\pK(X^{\lambda})$. 
\end{remark}

\begin{lemma}\label{lem:L_unif}
Let $X$ be a L{\'e}vy process in $\Rd$. Assume that
there are $t>0$ and $0\leq M<\infty$ such that for all $x\in\Rd$,
$$
G^0_t |q|(x)=\int_0^t  P_u|q|(x) \,du\leq M\,.
$$
Then there is a constant $0\leq M'<\infty$ independent of $q$ such that
\begin{align}\label{def:L_1_loc_uni}
\sup_{x} \int_{B(x,1)}|q(z)|\,dz\leq M'\,.
\end{align}
\end{lemma}
\begin{proof}
Let $\varphi\in C_0(\Rd)$ be such that $\varphi\geq 0$, $\varphi=1$ on $B(0,1)$ and $\int_{\Rd}\varphi(x)dx=N<\infty$. For $x_0\in\Rd$ we have, for $h\leq t$,
\begin{align*}
MN &\geq \int_0^h \int_{\Rd} P_u|q|(x)\varphi(x_0-x)\,dxdu=\int_0^h \int_{\Rd} \EE^0 |q(X_u+x)| \varphi (x_0-x)\,dxdu\\
&=\int_0^h \EE^0 \left[\int_{\Rd} |q(X_u+x)|\varphi(x_0-x)\,dx\right]du
=\int_0^h \EE^0 \left[\int_{\Rd} |q(z)|\varphi(X_u+x_0-z)\,dz\right]du\\
&=\int_0^h \int_{\Rd} |q(z)| P_u\varphi(x_0-z)\,dzdu\geq \int_0^h \int_{B(x_0,1)} |q(z)| P_u\varphi(x_0-z)\,dzdu\\
&\geq (\varepsilon/2) \int_{B(x_0,1)}|q(z)|\,dz\,,
\end{align*}
where $0<\varepsilon\leq h$ is such that $\|P_u\varphi -\varphi \|_{\infty}\leq 1/2$ for $u\leq \varepsilon$ (see \cite[Theorem 31.5]{MR1739520}).
\end{proof}

We write $q\in (L_{loc}^1)_{uni}(\Rd)$ if \eqref{def:L_1_loc_uni} holds. We collect basic properties of $\KK(X)$ and $\pK(X)$ for a L{\'e}vy process $X$ in $\Rd$.

\begin{proposition}\label{calInKato}
We have
\begin{enumerate}
\item $\mathcal{K}\subseteq \mathbb{K} \subseteq (L^1_{loc})_{uni}(\Rd)$ for every L{\'e}vy process\,, \label{incl:1}
\item $B(\Rd)\subseteq \KK$ for every L{\'e}vy process\,, \label{incl:2}
\item $B(\R^d)\subseteq \pK$ for every non-Poisson L{\'e}vy process\,, \label{incl:3}
\item $\mathcal{K}=\{0\}$ and $\KK=B(\Rd)$ for every compound Poisson process\,. \label{incl:4}
\end{enumerate}
\end{proposition}
\begin{proof}
The inclusion $\mathbb{K} \subseteq (L^1_{loc})_{uni}(\Rd)$ follows from Lemma \ref{lem:L_unif}. To complete {\it \ref{incl:1}}. we let $q\in \pK(X)$, which reads as \eqref{C1} for $X^{\lambda}$, $\lambda>0$. 
By Remark~\ref{H1H2} and Lemma \ref{lem:h1h2_sub}, 
\eqref{H2} holds for $X^{\lambda}$ and 
thus the result of Zhao on Figure~\ref{fig} implies that \eqref{C2} holds for $X^{\lambda}$,
i.e., $q\in\KK(X^{\lambda})=\KK(X)$ (see Remark~\ref{lem:KK_sub}).
Plainly, {\it\ref{incl:2}}. holds.
Now, let $X$ be non-Poisson. By Lemma \ref{lem:al_sure_0} we get $P_t(\{0\})=0$ for almost all $t>0$ and consequently $G^\lambda(\{0\})=0$. Further, since $G^{\lambda}(dx)$ is a finite measure, for $q\in B(\Rd)$ we have
$$\lim_{r\to\infty} \sup_{x\in\Rd} \int_{B_r} |q(x+z)|G^{\lambda}(dz)\leq \lim_{r\to 0^+}G^{\lambda}(B_r) \sup_{x\in\Rd} |q(x)| =G(\{0\}) \sup_{x\in\Rd} |q(x)|=0\,,$$
and {\it\ref{incl:3}}. holds.
Finally, if $X$ is a compound Poisson process, then $G^\lambda(\{0\})\geq(\lambda+\nu(\Rd))^{-1}>0$ and for every $r>0$
$$\sup_{x\in\Rd}\int_{B_r}|q(x+z)|G^\lambda(dz)\geq \sup_{x\in \Rd}|q(x)|(\lambda+\nu(\Rd))^{-1}.$$
Hence $q\in \mathcal{K}$ if and only if $q\equiv0$. Moreover,
$$
\sup_{x\in\Rd }\int_0^t P_u|q|(x)du\geq \sup_{x\in\Rd} |q(x)| \int_0^t e^{-\nu(\Rd)u}du\,,
$$
which proves {\it\ref{incl:4}}.
\end{proof}

\section{Main Theorems}\label{sec:thm}
In this section we consider a L{\'e}vy process $X$ in $\Rd$ and we pursue according to the cases of Section~\ref{sec:scheme}. 
Before that, we prove  Corollary~\ref{cor:time} directly from Theorem~\ref{thm:Gen_1}.

\begin{proof}[Proof of Corollary~\ref{cor:time}]
Consider a L{\'e}vy process $Y$ in $\R^{d+1}=\R\times \Rd$ defined by $Y_t=(t,X_t)$, $t\geq 0$, where $X$ is an arbitrary L{\'e}vy process in $\Rd$, $d\geq 1$.
Observe that for $(s,x)\in\R^{d+1}$ and a Borel set $B\subseteq \R^{d+1}$ we have $\PP^{(s,x)}(Y_u\in B)= \EE^x[\ind_{B}(s+u,X_u)]$, $u\geq 0$. Since for $Y$ $0$ is not regular for $\{0\}$ Theorem~\ref{thm:Gen_1} applies to $Y$. Finally, we use \eqref{starteqiv} taking into account that 
$\ind_{B_{d+1}((s,x),r)}(s+u,X_u)$, where $B_{d+1}(x,r)$ denotes a ball in $\R^{d+1}$, can be replaced with $\ind_{[0,r)}(u)\ind_{B(x,r)}(X_u)$ and that $e^{-\lambda u}$ is comparable with  one for $u\in[0,r)$.
\end{proof}

\subsection{Under \Hyp}\label{UnderH0}
In this subsection we consider a L{\'e}vy process $X$ satisfying \Hyp{}.

\begin{theorem}\label{prop:KatoGeneral}
For $d>1$ or $d=1$ under {\rm (A)} we have $\pK(X)=\KK(X)$.
\end{theorem}
\begin{proof}
By Proposition \ref{calInKato} we concentrate on $\KK(X)\subseteq \pK(X)$. Let $q\in\KK(X)=\KK(X^{\lambda})$, $\lambda>0$. This reads as \eqref{C2} for $X^{\lambda}$. Since $X$ is non-Poisson, by Remark \ref{H1H2} and Lemma \ref{lem:h1h2_sub} the hypothesis \eqref{H1} holds for $X^{\lambda}$. To obtain \eqref{C1} for $X^{\lambda}$, that is to prove $q\in\pK(X)$, it remains to verify \eqref{H3} for $X^{\lambda}$.
In view of Corollary \ref{lem:H3} it suffices to justify that $\{0\}$ is a polar set. For $d>1$ this is assured by Proposition \ref{polar}. For $d=1$ it is our assumption.
\end{proof}

From now on in this subsection we discuss the case of $d=1$.
For simplicity we recall from \cite[Theoreme 7, 1, 5, 6 and 8]{MR0368175} the following facts.

\begin{lemma}\label{lem:Bret}
Let $d=1$ and  $\int_{\R}\mathrm{Re}\left(\frac{1}{\lambda+\psi(z)}\right)dz<\infty$, $\lambda>0$.
Then $G^{\lambda}(dz)$ has a bounded density $G^{\lambda}(z)=k^{\lambda}\, h^{\lambda}(z)$, $z\in\R$, with respect to the  Lebesgue measure which is continuous on $\R\setminus\{0\}$. Further, $G^{\lambda}(z)$ is continuous at $0$ if and only if $0$ is regular for $\{0\}$ (i.e. $h^{\lambda}(0)=1$), and then $0<h^{\lambda}(z)\leq 1$ for $z\in \R$.
\end{lemma}

We investigate the properties of $G_t^{\lambda}(dz)$, $\lambda >0$, $t\in(0,\infty)$.

\begin{lemma}\label{lem:Glt}
Let $d=1$ and  $\int_{\R}\mathrm{Re}\left(\frac{1}{\lambda+\psi(z)}\right)dz<\infty$, $\lambda>0$. Then $G^\lambda_t(dz)$ has a bounded density $G_t^{\lambda}(z)$ with respect to the Lebesgue measure which is lower semi-continuous on $\R\setminus\{0\}$.
\end{lemma}
\begin{proof}
According to Lemma \ref{lem:Bret} we define $F^{\lambda}(z)=G^{\lambda}(z)$ on $\R\setminus \{0\}$ and $F^{\lambda}(0)=\limsup_{z\to 0} F^{\lambda}(z)$. Then $F^{\lambda}(z)$ is a density of $G^{\lambda}(dz)$. Since $G_t^{\lambda}(B)\leq G^{\lambda}(B)$ and
$G_t^{\lambda}(B)= G^{\lambda}(B)- e^{-\lambda t} \int_{\R} G^{\lambda}(B-z)P_t(dz)$,
$G_t^{\lambda}(dx)$ is absolutely continuous and its density $G_t^{\lambda}(x)$ can be chosen to satisfy
\begin{align}\label{GtLambda}
G_t^{\lambda}(x)= F^{\lambda}(x)- e^{-\lambda t} \int_{\R} F^{\lambda}(x-z)P_t(dz).
\end{align}
To prove the semi-continuity we observe that
for $x_0\in\R\setminus\{0\}$,
\begin{align*}
G_t^{\lambda}(x)=F^{\lambda}(x)- e^{-\lambda t}\( \int_{\R\backslash\{x_0\}} F^{\lambda}(x-z) P_t(dz)+F^\lambda(x-x_0)P_t(\{x_0\})\),
\end{align*}
and by the bounded convergence theorem
\begin{align*}
\liminf_{x\to x_0} G_t^{\lambda} (x) &= F^{\lambda}(x_0)-e^{-\lambda t}\( \int_{\R\backslash\{x_0\}}\lim_{x\to x_0} F^{\lambda}(x-z)P_t(dz)+\limsup_{x\to x_0}F^\lambda(x-x_0)P_t(\{x_0\})\)\\
&=G_t^{\lambda}(x_0)\,.
\end{align*}
\end{proof}

\begin{theorem}\label{bouVar1}
For $d=1$ under {\rm (B)} we have
$$\pK(X)=\KK(X)=\left\{q\colon\lim_{r\to 0^+}\sup_{x\in\R}\int_{B(x,r)}|q(z)|dz=0\right\}.$$
\end{theorem}
\begin{proof}
Without loss of generality we may and do assume that $\gamma_0>0$. Due to Proposition \ref{calInKato} and Lemma \ref{lem:Bret} (boundedness of the function $G^{\lambda}$) it remains to prove
$
\KK(X)\subseteq \{q\colon\lim_{r\to 0^+}\sup_{x\in\R}\int_{B(x,r)}|q(z)|dz=0\}
$.
By Remark~\ref{rem:var} we get $\PP^0(\lim_{u\to 0^+} u^{-1}X_u=\gamma_0)=1$.
Hence, there is $\varepsilon>0$ such that $\PP^0(|X_u-\gamma_0 u|<\gamma_0 u)\geq 1/2$ for $u\leq \varepsilon$.
This implies that for $t\leq \varepsilon$,
\begin{align*}
G^{\lambda}_{t}(0,2\gamma_0t]=\int^t_0 e^{-\lambda u}\PP^0(X_u\in(0,2\gamma_0t])du\geq\int^t_0 e^{-\lambda u}\PP^0(|X_u-\gamma_0 u|<\gamma_0 u)du\geq \frac{1-e^{-\lambda t}}{2\lambda}.
\end{align*}
Hence, $\sup_{z\in(0,2\gamma_0t]}G^\lambda_t(z)\geq \frac{1-e^{-\lambda t}}{\lambda t} \frac{1}{4\gamma_0} \geq \frac{1-e^{-\lambda \varepsilon}}{\lambda \varepsilon}\frac{1}{4\gamma_0}$. Since $G^\lambda_t(z)$ is lower semi-continuous on  $\R\setminus\{0\}$ there exist $0<a_t<b_t\leq \varepsilon$
 such that $G^\lambda_t(z)\geq\frac{1-e^{-\lambda \varepsilon}}{\lambda \varepsilon}\frac{1}{8\gamma_0}$ for $z\in(a_t,b_t)$. Now, let $q\in \KK(X)$.
We obtain for $t\leq \varepsilon$,
$$
\int_{\R} |q(x+z)|G_t^{\lambda}(dz)\geq
\frac{1-e^{-\lambda \varepsilon}}{8\lambda \varepsilon \gamma_0}\int^{b_t}_{a_t}|q(x+z)|dz.$$
Thus,
$$0=\lim_{t\to 0^+}\sup_{x\in\R}\int^{b_t}_{a_t}|q(x+z)|dz\geq  \lim_{r\to 0^+}\sup_{x\in\R}\int_{B(x,r)}|q(z)|dz\,.$$
\end{proof}

\begin{lemma}\label{lem:reg0}
Let $0$ be regular for $\{0\}$. There is $0<M_{G^{\lambda}}<\infty$ such that
\begin{equation}\label{WeakHarnack}
G^\lambda(x)\leq M_{G^{\lambda}}\, G^\lambda (y),\quad x,\,y\in\R,\,\,\, |x-y|\leq 1.
\end{equation}
Further, $G_t^\lambda(x)$ given by \eqref{GtLambda} is continuous on $\R$ and
\begin{align*}
G^\lambda_t(x)\leq G^\lambda(x)(\lambda t+||P_t f-f||_\infty)\,,\qquad f(x)=h^{\lambda}(-x)\in C_0(\R)\,.
\end{align*}
\end{lemma}

\begin{proof}
Let $F^{\lambda}$ be defined as in the proof of Lemma \ref{lem:reg0}. By Lemma \ref{lem:Bret} the functions $G^{\lambda}$ and $F^{\lambda}$ are equal and continuous on~$\R$.
Further, Lemma \ref{lem:tau} implies that the function $h^\lambda(x)= G^{\lambda}(x)/ k^{\lambda}=\EE^0 e^{-\lambda T_{\{x\}}}$ is in $C_0(\R)$. Since $h^{\lambda}(x+y)\geq h^{\lambda}(x)h^{\lambda}(y)$, $x,y\in\R$ (see remarks after \cite[Lemma 2]{MR0368175}), we get $$\frac{G^\lambda(x-z)}{G^\lambda(x)}=\frac{h^\lambda(x-z)}{h^\lambda(x)}\geq h^\lambda(-z)\,.$$
By positivity and continuity of $h^\lambda$ we obtain \eqref{WeakHarnack} with $M_{G^{\lambda}}=\sup_{|z|\le 1} 1/[h^\lambda(z)]<\infty$.
Eventually, by  \eqref{GtLambda},
\begin{align*}
G^\lambda_t(x)&=G^\lambda(x)\(1-e^{-\lambda t}+e^{-\lambda t}\int_{\R}\(1-\frac{G^\lambda(x-z)}{G^\lambda(x)}\)P_t(dz)\)\\
&\leq G^\lambda(x)\(\lambda t+\int_{\R}\(h^\lambda(0)-h^\lambda(-z)\)P_t(dz)\).
\end{align*}
\end{proof}

\begin{theorem}\label{thm:0regH0}
For $d=1$ under {\rm (C)} we have $\pK(X) \subsetneq \KK(X)$,
$$\pK(X)=\left\{q\colon \lim_{r\to 0^+} \sup_{x\in\R} \int_{B(x,r)} |q(z)|\,dz =0 \right\},$$
and
$$\KK(X)=(L_{loc}^1)_{uni}(\R)=\left\{q\colon \sup_{x\in\R}\int_{B(x,1)}|q(z)|dz<\infty\right\}.$$
\end{theorem}
\begin{proof}
For $\pK(X)$ we just observe that $G^{\lambda}(z)$ is bounded and $G^{\lambda}(z)\geq \varepsilon>0$ if $|z|\leq 1$.
Now, we describe $\KK(X)$.
The condition $q\in (L_{loc}^1)_{uni}(\R)$ is necessary by Lemma~\ref{lem:L_unif}. We show that it is sufficient.
Let $\lambda>0$ and denote $c_t=\lambda t+||P_t f - f||_\infty$, where $f(x)=h^{\lambda}(-x)=\EE e^{-\lambda T_{\{-x\}}}$. By Lemma \ref{lem:reg0}
\begin{align}
\int_{\R} |q(x+z)|&G^\lambda_t(dz)\leq c_t \int_{\R}|q(x+z)|G^\lambda(z)dz
=c_t\sum^\infty_{k=-\infty}\int^{k+1/2}_{k-1/2}|q(x+z)|G^\lambda(z)dz\nonumber \\
&\leq c_t\, M_{G^{\lambda}}
\sum^\infty_{k=-\infty}G^\lambda(k)\int^{k+1/2}_{k-1/2}|q(x+z)|dz
\leq c_t\, M_{G^{\lambda}} \sup_{x\in\R}\int_{B(x,1)}|q(z)|dz
\sum^\infty_{k=-\infty}G^\lambda(k) \nonumber \\
&\leq  c_t\, (M_{G^{\lambda}})^2 \lambda^{-1}\sup_{x\in\R}\int_{B(x,1)}|q(z)|dz. \label{ineq:0reg}
\end{align}
Since $f\in C_0(\R)$ we get $c_t\to 0$ as $t\to 0^+$.
\end{proof}

\subsection{Without \Hyp}
In this subsection we assume that \Hyp{} does not hold. In view of Proposition~\ref{calInKato} we assume that
$d>1$ and $X$ is given by \eqref{eq:descr_X}. 
We use results of Section \ref{UnderH0} and analyze the cases (A'), (B') and (C').

\begin{theorem}
Under {\rm (A')} we have $\pK(X)=\KK(X)$.
\end{theorem}

\begin{proof}
Following the proof of Theorem \ref{prop:KatoGeneral} it remains to show that $\{0\}$ is polar for the process~$X$.
This is assured by Corollary~\ref{lem:naX}.
\end{proof}

We proceed to the remaining cases. 
The transition kernel of $X$ equals
\begin{align*}
P_t(dx)=P^Z_t*\sum^\infty_{n=0}e^{-t\nu^Y(\Rd)}\frac{t^n(\nu^{Y})^{*n}}{n!}(dx)\,.
\end{align*}
The characteristic exponent $\psi$ of $X$ can be written as $\psi=\psi^Y + \psi^Z$.
We note that $\psi^Z(z)=\psi^Z(v)$ for $z=v+w\in\Rd$, $v\in V$, $w\in V^{\bot}$.
For $\lambda>0$, $t\in(0,\infty]$ and $n\in\NN$ we define
\begin{align*}
\GZln_t(dv)= \int_0^t u^n e^{-\lambda u} P_u^Z(dv)\,du\,.
\end{align*}
We investigate $n$-moment $\lambda$-potentials $\GZln(dv)=\GZln_{\infty}(dv)$ and truncated $\lambda$-potentials $\GZL{\lambda}_t(dv)=\GZLN{\lambda}{0}_t(dv)$ of $Z$.
We also write $\GZL{\lambda}(dv)=\GZLN{\lambda}{0}_{\infty}(dv)$
for $\lambda$-potentials of $Z$.
The measures $\GZL{\lambda},\, \GZL{\lambda}_t,\, \GZln$ are concentrated on $V$.
Observe that
\begin{align}\label{Glambda_form}
G^{\lambda}(dx)= \sum_{n=0}^{\infty} \frac{1}{n!}\, \GZLn{\lambda+\nu^Y(\Rd)}* (\nu^Y)^{*n}(dx)\,.
\end{align}
We reformulate Lemma \ref{lem:Glt} and Lemma \ref{lem:reg0} in view of Remark \ref{obrot}.

\begin{lemma}\label{lem:wlasnosci_Z}
Let $\int_{V}\mathrm{Re} \left(\frac{1}{\lambda + \psi^Z(v)} \right)dv<\infty$, $\lambda>0$. Then $\GZL{\lambda}_t (dv)$ has a bounded density $\GZL{\lambda}_t (v)$ with respect to the Lebesgue measure on $V$ which is lower semi-continuous on $V\setminus\{0\}$.
If $0$ is regular for $\{0\}$ for Z then there is $0<M_{\GZL{\lambda}}<\infty$ such that
\begin{equation*}
\GZL{\lambda} (v)\leq M_{\GZL{\lambda}}\, \GZL{\lambda} (v'),\quad v,\,v'\in V,\,\,\, |v-v'|\leq 1,
\end{equation*}
$\GZL{\lambda}_t (v)$ is continuous on $V$ and
\begin{align*}
\GZL{\lambda}_t (v)\leq \GZL{\lambda}(v)(\lambda t+||P_t^Z f-f||_\infty)\,,\qquad f(v)\in C_0(V)\,.
\end{align*}
\end{lemma}

\begin{lemma}\label{lem:Gtilde}
Let $\int_{V}\mathrm{Re} \left(\frac{1}{\lambda + \psi^Z(v)} \right)dv<\infty$, $\lambda>0$.
Then $\GZln(dv)$ has a density $\GZln(v)$ with respect to the Lebesgue measure on $V$, and
\begin{equation}\label{Gtilde}
\GZln(v)\leq \frac{n!}{\lambda^{n}}\int_V\mathrm{Re}\left(\frac{1}{\lambda+\psi^Z(u)}\right)du.
\end{equation}
\end{lemma}

\begin{proof}
By Remark \ref{obrot} we assume that $V=\R$ and we observe that the Fourier transform of $\GZln$ equals
$$
\int^\infty_0t^ne^{-\lambda t} e^{-t\psi^Z(\xi)}dt=\frac{n!}{[\lambda+\psi^Z(\xi)]^{n+1}}\,,\qquad \xi\in \R\,.
$$
Since $\mathrm{Re}(1/z)=\mathrm{Re}(\bar{z})/|z|^2$ and $\mathrm{Re}[\psi]\geq0$ we obtain
$$
\frac{1}{|\lambda+\psi^Z(\xi)|^{n+1}}\leq \lambda^{-n+1}\frac{1}{|\lambda+\psi^Z(\xi)|^{2}}\leq \lambda^{-n} \mathrm{Re}\left(\frac{1}{\lambda+\psi^Z(\xi)}\right).
$$
This implies that the Fourier transform is integrable and \eqref{Gtilde} follows by the inversion formula.
\end{proof}

\begin{lemma}\label{lem:Glambda}
Let $\int_{V}\mathrm{Re} \left(\frac{1}{\lambda + \psi^Z(v)} \right)dv<\infty$, $\lambda>0$.
Then
\begin{align*}
\sup_{x\in\Rd} \left(\int_{B(0,r)}|q(x+z)| G^{\lambda}(dz)\right) \leq   \sup_{x\in\Rd} \left( \int_{B(0,r)\cap V} |q(x+v)| \,dv\right) C \big[1+ \nu^Y(\Rd)/\lambda \big]\,,
\end{align*}
where $dv$ is the  one-dimensional Lebesgue measure on $V$ and $C=\int_{V}\mathrm{Re}\left(1/[\lambda +\nu^Y(\Rd)+\psi^Z(u)] \right)du$.
\end{lemma}
\begin{proof}
By \eqref{Glambda_form} and \eqref{Gtilde} we have
\begin{align*}
\int_{B(0,r)}|q(x+z)| G^{\lambda}(dz)=& \sum_{n=0}^{\infty}\frac{1}{n!} \int_{\Rd} \left( \int_V \ind_{B(0,r)}(v+w) |q(x+v+w)| \GZLn{\lambda+\nu^Y(\Rd)}(dv)\right)(\nu^{Y})^{*n}(dw)\\
\leq & \sup_{x,w\in\Rd} \left(\int_V \ind_{B(0,r)}(v+w) |q(x+v+w)| \,dv\right) \sum_{n=0}^{\infty} C \left( \frac{\nu^Y(\Rd)}{\lambda+\nu^Y(\Rd)}\right)^n\,,
\end{align*}
and
\begin{align*}
\sup_{x,w\in\Rd} \left(\int_V \ind_{B(0,r)}(v+w) |q(x+v+w)| \,dv\right)
= \sup_{x,w\in\Rd} \left(\int_{B(-w,r)\cap V} |q(x+v)| \,dv\right)&\\
= \sup_{x\in\Rd,\,w\in V}\left( \int_{B(-w,r)\cap V} |q(x+v)| \,dv\right)
= \sup_{x\in\Rd} \left( \int_{B(0,r)\cap V} |q(x+v)| \,dv\right)&,
\end{align*}
where the last equality follows by the translation invariance of the Lebesgue measure on $V$. This ends the proof.
\end{proof}

\begin{theorem}\label{prop:WH0_bouVar1}
Under {\rm (B')} we have
$$
\pK(X)=\KK(X)=\left\{q\colon \lim_{r\to 0^+}\sup_{x\in\Rd}\int_{B(0,r)\cap V}|q(x+v)|dv=0\right\},
$$
where $dv$ is the  one-dimensional Lebesgue measure on $V$.
\end{theorem}

\begin{proof}
Lemma \ref{lem:Glambda} gives
$
\{q\colon \lim_{r\to 0^+}\sup_{x\in\Rd}\int_{B(0,r)\cap V}|q(x+v)|dv=0\} \subseteq \pK(X)
$.
By Proposition~\ref{calInKato} it suffices to show
$
\KK(X)\subseteq\{q\colon \lim_{r\to 0^+}\sup_{x\in\Rd}\int_{B(0,r)\cap V}|q(x+v)|dv=0\}
$.
Since for $t>0$ and $x\in\Rd$ we have
\begin{align*}
\int^t_0 P_u|q|(x)\,du\geq \int_0^t \int_{\Rd} |q(x+z)|\, e^{-u\nu^Y(\Rd)} P_u^Z(dz)\,du= \int_{\Rd\cap V} |q(x+v)|\, \GZL{\,\nu^Y(\Rd)}_t(dv),
\end{align*}
the inclusion
follows by adapting the proof of Theorem~\ref{bouVar1} to the one-dimensional process $Z$ with the support of Lemma~\ref{lem:wlasnosci_Z} and Remark~\ref{rem:var}.
\end{proof}

\begin{theorem}\label{thm:WH0_0reg}
Under {\rm (C')} we have $\pK(X)\subsetneq\KK(X)$,
$$\pK(X)=\left\{q\colon \lim_{r\to 0^+}\sup_{x\in\Rd}\int_{B(0,r)\cap V}|q(x+v)|\,dv=0\right\},$$
and
$$\KK(X)=\left\{q\colon\sup_{x\in\Rd}\int_{B(0,1)\cap V}|q(x+v)|\, dv<\infty\right\},$$
where $dv$ is the one-dimensional Lebesgue measure on $V$.
\end{theorem}

\begin{proof}
The condition postulated for the description of $\pK(X)$ is sufficient by Lemma \ref{lem:Glambda}. Next, by Remark \ref{obrot} and Lemma \ref{lem:Bret}
the $\lambda$-potential kernel of $Z$, that is $\GZL{\lambda}(dv)=\GZLN{\lambda}{0}(dv)$, has a density $\GZL{\lambda}(v)$ with respect to the Lebesgue measure on $V$, such that $\GZL{\lambda}(v)\geq \varepsilon>0$ if $v\in B(0,1)\cap V$ ($\varepsilon$ may depend on $\lambda$). Thus,
\begin{align*}
\int_{B(0,r)} |q(x+z)| G^{\lambda}(dz)\geq \int_{B(0,r)\cap V} |q(x+v)| \GZL{\lambda+\nu^Y(\Rd)}(dv)\geq \varepsilon \int_{B(0,r)\cap V} |q(x+v)|\,dv\,,
\end{align*}
which proves the necessity.
Further, the necessity of the condition proposed to describe $\KK(X)$ follows from Remark~\ref{obrot}, Lemma~\ref{lem:L_unif} and
\begin{align*}
\int_0^t P_u|q|(x)du
\geq \int_0^t \int_{\Rd\cap V} |q(x+v)|\, e^{-u\nu^Y(\Rd)} P_u^Z (dv)du
\geq e^{-t\nu^Y(\Rd)} \int_0^t \int_{\Rd\cap V} |q(x+v)|\, P_u^Z (dv)du.
\end{align*}
For the sufficiency we partially follow the proof of Theorem \ref{thm:0regH0}.
Note that
$
\int_0^t u^n e^{-\lambda u} P_u^Z(dv)\,du\leq t^n \GZL{\lambda}_t(dv)
$
which gives
\begin{align*}
G_t^{\lambda}(dx)\leq \sum_{n=0}^{\infty} \frac{t^n}{n!}\, \GZL{\lambda+\nu^Y(\Rd)}_t * (\nu^Y)^{*n} (dx)\,.
\end{align*}
Thus by Lemma \ref{lem:wlasnosci_Z} and adaptation of \eqref{ineq:0reg} we have with $c_t=(\lambda+\nu^Y(\Rd))t + ||P_t^Z f-f ||_{\infty}$,
\begin{align*}
\int_{\Rd}&|q(x+z)| G_t^{\lambda} (dz) \leq \sum_{n=0}^{\infty}\frac{t^n}{n!} \int_{\Rd} \left(\int_V |q(x+v+w)|\GZL{\lambda+\nu^Y(\Rd)}_t (dv)\right) (\nu^Y)^{*n}(dw) \\
&\leq \left(c_t \left( M_{\GZL{\lambda+\nu^Y(\Rd)}}\right)^2 (\lambda+\nu^Y(\Rd))^{-1} \sup_{x\in\Rd} \int_{B(0,1)\cap V} |q(x+v)|dv \right) \sum_{n=0}^{\infty} \frac{t^n}{n!} \int_{\Rd} (\nu^Y)^{*n}(dw)\,,
\end{align*}
which ends the proof.
\end{proof}

\subsection{Zero-potential kernel}

In the previous sections and subsections we have already used measures
$G_t^{\lambda}$,  $\lambda\geq 0$, $t\in(0,\infty]$.
Below we present additional sufficient assumptions on a L{\'e}vy process $X$ under which $G^0=G^0_{\infty}$ can be used to describe  $\KK(X)$.
The condition we want to analyze now is $q\in \pK^0(X)$ defined by
\begin{align}\label{cond_G_infty^0}
\lim_{r\to 0^+}\left[\sup_{x\in\Rd} \int_{B(0,r)} |q(z+x)|G^0(dz)\right]=0\,.
\end{align}
Since $G^{\lambda}(dz)\leq G^0(dz)$, \eqref{cond_G_infty^0} implies $q\in\pK(X)$  and thus $\pK^0(X)\subseteq \pK(X)\subseteq \KK(X)$ by Proposition \ref{calInKato}.
Our aim is to obtain the equivalence, i.e., the implication from $q\in \KK(X)$ to \eqref{cond_G_infty^0}, and this is the subcase of $\pK(X)=\KK(X)$.
We will assume that $X$ is transient and $\{0\}$ is polar (in Theorem~\ref{prop:G_unb_b} polarity 
follows
implicitly from other assumptions). 
The transience
is necessary, otherwise $G^0(dz)$ is
locally unbounded 
(see \cite[Theorem 35.4]{MR1739520}) and non-zero constant functions do not belong to $\pK^0(X)$,
which shows $\pK^0(X)\subsetneq \KK(X)$.
The polarity of $\{0\}$ assures $\pK(X)=\KK(X)$. Moreover, if $\{0\}$ is not polar, the class $\KK(X)$ is explicitly described by our previous theorems.
Both, transience and polarity of $\{0\}$
are to some extent encoded in the characteristic exponent $\psi$
(see \cite[Remark~37.7]{MR1739520} and Section~\ref{sec:scheme}).
Finally, we note that $q\in\pK^0(X)$ is equivalent to \eqref{C1} and $q\in\KK(X)$ to \eqref{C2}. Thus according to Figure~\ref{fig} and Remark~\ref{H1H2}, we focus on showing \eqref{H3} for $X$.

\begin{remark}\label{rem:point_X_trans}
If $X$ is transient, then
we have
\begin{align}\label{eq:point_X_trans}
\lim_{r\to 0^+} \PP^0(T_{\overline{B}(x,r)}<\infty)=\PP^0(T_{\{x\}}<\infty)\,,\quad x\in\Rd\,.
\end{align}
Such statement is not true in general, but here it follows from
$\PP^0(T_{\overline{B}(x,r)}<\infty)=\PP^0(T_{\overline{B}(x,r)}<\infty, T_{\{x\}}<\infty)+\PP^0(T_{\overline{B}(x,r)}<\infty, T_{\{x\}}=\infty)$, Lemma \ref{lem:point} and $\lim_{t\to\infty} |X_t|=\infty$ $\PP^0$ a.s.
\end{remark}
We say that a measure $G^0(dz)$ tends to zero at infinity if $\lim_{|x|\to \infty}\int_{\Rd} f(z+x) G^0(dz)=0$ for 
all 
$f\in C_c(\Rd)$.
Under certain assumptions on the {\it group} of the L{\'e}vy process \cite[Definition~24.21]{MR1739520} $G^0(dz)$ tends to zero for every transient $X$ if $d\geq 2$. 
The case $d=1$ is more complicated.
See \cite[Exercise~39.14]{MR1739520}
and
Remark~\ref{special_to_zero}.

\begin{lemma}\label{lem:G_to_zero}
Let $X$ be transient. If $G^0(dz)$ tends to zero at infinity then
$$h_3(X)=\sup_{x\neq 0} \PP^{0} (T_{\{x\}}<\infty)\,.$$
\end{lemma}

\begin{proof}
The statement follows by the same proof as for Proposition  \ref{lem:h3_sub_Levy} but with $\lambda=0$ and a version of Lemma \ref{lem:ZhH3} for $\lambda=0$.
To prove the latter one we also repeat its proof with functions $f_r$ extended to $\lambda=0$, i.e., $f_r(x)=\PP^0(T_{\overline{B}(x,r)}<\infty)$ up to a moment when $a>0$ and a sequence $\{x_n\}$ such that $f_{1/n}(x_n)>a-\varepsilon$ are chosen. The rest of the proof easily applies with \eqref{eq:point_X_trans} in place of Lemma \ref{lem:point} as soon as we can show that $\{x_n\}$ is bounded. To this end assume that the sequence is unbounded. Since $f_r(x)=\PP^{y}(T_{\overline{B}(x+y,r)}<\infty)$, $r>0$, $y\in\Rd$,
for $r\in(0,1]$ and  $|x-x_n|<1$ we have
\begin{align}\label{ineq:contr}
a-\varepsilon <f_r(x_n)=\PP^{-x}(T_{\overline{B}(x_n-x,r)}<\infty)\leq \PP^{-x}(T_{\overline{B}(0,2)}<\infty)=f_{2}(x)\,,
\end{align}
Next, by \cite[Theorem 42.8 and Definition 41.6]{MR1739520} for $g\in C_c(\Rd)$ such that $\ind_{B(0,1)}\leq g$ we get
\begin{align*}
\int_{\Rd} g(x_n-x)f_2(x)\,dx=\int_{\Rd}\left[ \int_{\Rd} g(-v+w+x_n) G(dv)\right]m_{\overline{B}(0,2)}(dw)\quad \xrightarrow {n\to \infty} 0\,,
\end{align*}
since $m_{\overline{B}(0,2)}(dw)$ is finite and supported on $\overline{B}(0,2)$ and $G(dv)$ tends to zero at infinity.
This contradicts \eqref{ineq:contr} and ends the proof.
\end{proof}

\begin{theorem}\label{prop:G_to_zero}
Let $X$ be transient, $\{0\}$ be polar and $G^0(dz)$ tend to zero at infinity. Then $q\in\KK(X)$ if and only if \eqref{cond_G_infty^0} holds, i.e., $\pK^0(X)=\pK(X)=\KK(X)$.
\end{theorem}

In the next result we improve \cite[Lemma 5]{MR1132313} and we cover some cases when $G^0(dz)$ may not tend to zero at infinity.
\begin{theorem}\label{prop:G_unb_b}
Let $X$ be transient and let $G^0(dz)$ have a density $G^0(z)$ with respect to the Lebesgue measure which is unbounded and bounded on $|z|\geq r$ for every $r>0$. Then $\pK^0(X)=\pK(X)=\KK(X)$.
\end{theorem}

\begin{proof}
We note that
the polarity of $\{0\}$ follows by our assumptions (see \cite[Theorem~41.15 and~43.3]{MR1739520}).
By \cite[Proposition~42.13 and Definition~42.9]{MR1739520} for $r>0$ we have
$$
\PP^x(T_{B(0,r)}<\infty)=\int_{\overline{B}(0,r)}G^0(y-x)\,m_{B(0,r)}(dy),\quad x \in\Rd .
$$
Next, for $u>0$, $|x|\geq u$ and $0<r<u/2$ we obtain,
$$
\PP^x(T_{B(0,r)}<\infty)\leq \Big[ \sup_{|y|\geq u/2 }G^0(y)\Big] C\!  \left(B(0,r)\right)\,,$$
where $\rm C(\cdot)$ stands for capacity.
By \cite[Proposition 42.10 and (42.20)]{MR1739520} and Remark \ref{rem:point_X_trans} we have $\lim_{r\to 0^+} C(B(0,r))=C(\{0\})$ (see also \cite[Proposition 8.4]{MR0346919}). This gives
\begin{align*}
h_3(X)
= \sup_{u>0} \inf_{r>0} \sup_{|x|\geq u} \PP^x ( T_{B(0,r)}<\infty)
&\leq  \sup_{u>0} \Big[ \sup_{|y|\geq u/2 }G(y)\Big] \inf_{0<r<u/2} C\!  \left(B(0,r)\right)\\
&= \sup_{u>0} \Big[ \sup_{|y|\geq u/2 }G(y)\Big] C(\{0\})\,.
\end{align*}
Finally, since $\{0\}$ is polar, by \cite[Theorem 42.19]{MR1739520} we have $C(\{0\})=0$ and so \eqref{H3} holds with $h_3(X)=0$.
\end{proof}

\section{Further discussion and applications}

In this section we give additional results for isotropic unimodal L{\'e}vy processes concerning (the implication) $\pK(X)\subseteq \KK(X)$, we apply general results to a subclass of subordinators and we present examples.

We recall from \cite{BGR1} the definition of weak scaling.
Let $\lt\in [0,\infty)$ and
 $\phi$ be a non-negative non-zero  function on $(0,\infty)$.
We say that
$\phi$ satisfies {the} weak lower scaling condition (at infinity) if there are numbers
$\la\in \R
$
and  $\lC
\in(0,1]$,  such that
\begin{align*}
 \phi(\LL \theta)\ge
\lC\LL^{\,\la} \phi(\theta)\quad \mbox{for}\quad \LL\ge 1, \quad\theta>\lt.
\end{align*}
In short we say that $\phi$ satisfies WLSC($\la, \lt,{\lC}$) and write $\phi\in\WLSC{\la}{ \lt}{\lC}$.
Similarly, we consider
 $\ut\in [0,\infty)$.
The weak upper scaling condition holds if there are numbers $\ua \in \R$
and $\uC{\in [1,\infty)}$ such that
\begin{align*}
\phi(\LL\theta)\le
\uC\LL^{\,\ua} \phi(\theta)\quad \mbox{for}\quad \LL\ge 1, \quad\theta> \ut.
\end{align*}
In short, $\phi\in\WUSC{\ua}{ \ut}{\uC}$.

\subsection{Isotropic unimodal L\'{e}vy processes}

A measure on $\Rd$ is called isotropic
unimodal, in short, unimodal, if it is absolutely continuous on $\Rd\setminus \{0\}$
with a radial non-increasing density (such measures may have an atom at the origin). A L\'evy process $X$ is called (isotropic) unimodal if all of its one-dimensional distributions $P_t(dx)$ are
unimodal.
Unimodal pure-jump
L\'evy processes are characterized in \cite{MR705619}
by isotropic unimodal L\'evy measures
$\nu(dx)=\nu(x)dx=\nu(|x|)dx$.
The distribution of $X_t$ has a radial non-increasing density $p(t,x)$ on $\Rd\setminus\{0\}$, and atom at the origin, with mass $\exp[-t\nu(\Rd)]$
(no atom if $\psi$ is unbounded).

For a continuous non-decreasing function $\phi : [0,\infty) \to [0,\infty)$, {such that $\phi(0)=0$,} we let $\phi(\infty)=\lim_{s\to \infty}\phi(s)$ and
we define
the generalized left inverse $\phi^-: {[0,\infty) \to [0, \infty]}$,
\begin{align*}
\phi^-(u) = \inf \{{s}\geq 0 \colon \phi (s)= u\}
=\inf \{{s}\geq 0 \colon \phi (s)\geq u\},\quad 0\leq u<\infty,
\end{align*}
with the convention that $\inf \emptyset = \infty$. The function is increasing and c\`agl\`ad where finite.
Notice that $\phi(\phi^-(u))=u$ for $u\in [0,\phi(\infty)]$ and $\phi^-(\phi(s))\leq s$ for $s\in [0,\infty)$. Moreover, by the continuity of $\phi$  we have $\phi^-(\phi(s)+\varepsilon)>s$ for $\varepsilon>0$ and $s\in [0,\infty)$.
We also define  $f^*(u)=\sup_{|x|\leq u}|f(x)|$ for $f\colon\Rd\to\R$.

In view of general results for Schr{\"o}dinger perturbations \cite[Theorem 3]{MR3000465} and the so-called 3G type inequalities \cite[(40) and Corollary 11]{MR2457489} it is desirable to have the following results which extend \cite[Theorem 1.28]{MR1772266} and \cite[Proposition 4.3]{MR3200161} (see also \cite[Remark 2]{MR3000465}).

\begin{proposition}\label{cor:Kato1}Let $X$ be unimodal.
For $t_0\in(0,\infty]$, $r>0$ and $0<t< t_0$,
$$\sup_{x\in\Rd}\int^t_0P_u|q|(x)du\leq \left(1+\frac{t}{|B(0,1/2)|r^d G_{t_0}^0(r)}\right)
\left[ \sup_{x\in\Rd}\int_{B(x,r)}|q(z)| G_{t_0}^0(z-x)dz\right]\,,$$
where $G_{t_0}^0(z)=\int_0^{t_0} p(u,z)\,du$, $z\in\Rd$, and $G_{t_0}^0(r)=G_{t_0}^0(x)$, $|x|=r$.
\end{proposition}
\begin{proof}
We use \cite[Lemma 4.2]{MR3200161} with $k(x)=\int^t_0p (u,x)du$ and $K(x)=G_{t_0}^0(x)$.
\end{proof}

In what follows we assume that $d\geq 3$ and that the L\'{e}vy-Khintchine exponent $\psi$ is unbounded. Then since $X$ is (isotropic) unimodal by \cite[Theorem~37.8]{MR1739520} it is transient and the measure $G^0(dz)$ has a radially non-increasing density $G^0(z)$. This density is unbounded (see \cite[Theorem~43.9 and Theorem~43.3]{MR1739520}). Thus Theorem~\ref{prop:G_unb_b} applies and $\pK^0(X)=\pK(X)=\KK(X)$. Under additional assumptions we investigate this relations. 
\begin{remark}
Below we use the result of \cite[Theorem 3]{G} which says that if $X$ is unimodal and $d\geq 3$ we always have $G^0(x)\leq C /(|x|^d \psi^*(|x|^{-1}))$, $x\in\Rd$, for some $C>0$. If~additionally $\psi\in {\rm WLSC}(\la,\lt,\lC)$, $\la>0$, then $c /(|x|^d \psi^*(|x|^{-1}))\leq G^0(x)$ for $|x|$ small enough and some $c>0$.
\end{remark}

\begin{corollary}\label{cor:uni_upper}
Let $d\geq 3$, $X$ be unimodal
with $\psi\in {\rm WLSC}(\la,\lt,\lC)$, $\la>0$. There exist constants $C=C(d,\la,\lC)$ and $b=(d,\la,\lC)$ such that for any $0<t<1/\psi^*(\lt/b)$ and $q:\Rd\to\R$,
$$\sup_{x\in\Rd}\int^t_0P_u |q|(x)du\leq C
\sup_{x\in\Rd}\int_{B(x,1/(\psi^*)^-(1/t))}|q(z)|G^0(z-x)dz.$$
\end{corollary}
\begin{proof}
We let $t_0=\infty$ in Proposition \ref{cor:Kato1}.
For $0<t<\infty$ we take $r= 1/(\psi^*)^-(1/t)>0$. Since $\psi^*(r^{-1})=1/t$ by \cite[Theorem 3]{G}
$
r^d G^0(r)\geq c /\psi^*(r^{-1})= c t
$
if $1/(\psi^*)^-(1/t)\leq b/\lt$ for some constant $c>0$.
The last holds if $t<1/\psi^*(\lt/b)$.
\end{proof}

\begin{lemma}\label{lem:uni_low}
Let
$d\geq 3$, $X$ be unimodal and $\psi\in {\rm WLSC}(\la,\theta,\lC)\cap {\rm WUSC} (\ua,\theta,\uC)$, $\la,\ua\in(0,2)$. Then there exist constants $c=c(d,\la,\ua,\lC,\uC)$ and $a=(d,\la,\ua,\lC,\uC)$ such that for any $0<t<1/\psi^*(\theta/a)$ and $q:\Rd\to\R$,
$$\sup_{x\in\Rd}\int^t_0 P_u|q|(x)du\geq c
\sup_{x\in\Rd}\int_{B(x,1/(\psi^*)^-(1/t))}|q(z)|G^0(z-x)dz.$$
\end{lemma}
\begin{proof}
Let $x\in\Rd$ be such that $|x|<1/(\psi^*)^-(1/t)$, which gives $1/\psi^*(|x|^{-1})\leq t$. Further, since $t<1/\psi^*(\theta/a)$ implies $1/(\psi^*)^-(1/t)<a/\theta$ we get $|x|<a/\theta$ and also $u\,\psi^*(\theta/a)<1$ if $u< 1/\psi^*(|x|^{-1})$. Then \cite[Theorem 21 and Lemma 17]{BGR1} ($r_0=a$) yield
\begin{align*}
\int_0^t p(u,x)du
\geq \int_0^{1/\psi^*(|x|^{-1})} p(u,x)du
\geq c^* \int_0^{1/\psi^*(|x|^{-1})} \frac{u \psi^*(|x|^{-1})}{|x|^d}du
= \frac{c^*}{2|x|^d \psi^*(|x|^{-1})}\,.
\end{align*}
Finally, we apply  \cite[Theorem 3]{G} to obtain
$$
\int_0^t p(u,x)du\geq c \,G^0(x)\,,\qquad {\rm for}\quad |x|<1/(\psi^*)^-(1/t)\,.
$$
\end{proof}

\subsection{Subordinators}\label{sec:sub}

Let $X$ be a subordinator (without killing) with the Laplace exponent $\phi$.
Then $\phi$ is a Bernstein function (in short BF) with zero killing term. Two important subclasses of BF are special Bernstein functions (SBF) and complete Bernstein functions (CBF). We refer the reader to \cite{MR2978140} for definitions and an overview.
Since the cases when $\phi$ is bounded (equivalently $X$ is a compound Poisson process) or when $X$ has a non-zero drift $\gamma_0$, are completely described by Theorem \ref{calInKato} and Theorem \ref{bouVar1}, we assume that
\begin{itemize}
\item[(S1)] $\phi$ is unbounded ($X$ is non-Poisson)
 and $\gamma_0=0$.
\end{itemize}
Note that for $d=1$ if a L{\'e}vy process is non-Poisson and $A=0$, $\gamma_0=0$, $\int_{\R}(|x|\wedge 1) \nu(dx)<\infty$,
then we are in the case (A) of Section~\ref{sec:scheme} (see Remark~\ref{opisH0}).
Thus by 
Theorem~\ref{prop:KatoGeneral}
the following is true for subordinators.
\begin{remark}\label{sub}
If $X$ satisfies (S1), then 
$\{0\}$ is polar and $\pK(X)=\KK(X)$.
\end{remark}
\noindent
We impose further assumptions on the exponent $\phi$ to study $G^{\lambda}(dz)$, $\lambda\geq 0$, and describe its behaviour near the origin:
\begin{itemize}
\item[(S2)] $a+\phi\in {\rm SBF}$ for some $a\geq 0$ (see \cite[Remark 11.21]{MR2978140}),
\item[(S3)] $\dfrac{\phi'}{\phi^2} \in {\rm WUSC}(-\beta,\ut,\uC)$, $\beta>0$.
\end{itemize}
We shall mention that (S2) is always satisfied if $\phi\in$ CBF. Indeed, 
if $\phi\in {\rm CBF}$, then $a+\phi\in {\rm CBF}$, $a\geq0$, and ${\rm CBF}\subset {\rm SBF}$.

\begin{remark}\label{special_to_zero}
Recall that $X$ is a subordinator without killing, i.e., $\phi\in {\rm BF}$ with zero killing term.
Note that $U(dz)=G^a(dz)$ is a potential kernel of (possibly killed) subordinator $S=X^a$, see \cite[(5.2)]{MR2978140}.
The Laplace exponent of $S$ equals $a+\phi$,
thus by \cite[Theorem 11.3, formulas (11.9) and Corollary 11.8]{MR2978140} we have
\begin{enumerate}
\item[]
\begin{enumerate}
\item under {\rm (S2)}, the measure $G^a(dz)$ is absolutely continuous with respect to the Lebesgue measure if and only if $\nu(0,\infty)=\infty$ ($X$ is non-Poisson) or $\gamma_0>0$,
\item under {\rm (S1)} and {\rm (S2)}, the density $G^a(z)$ of $G^a(dz)$ satisfies: $G^a(z)=0$ on $(-\infty,0]$, $G^a(z)$ is finite, positive and non-increasing on $(0,\infty)$, and $\lim_{z\to 0^+} G^a(z)=\infty$,
\item under {\rm (S2)} with $a=0$, $G^0(dz)$ tends to zero if and only if $\int_1^{\infty} x \nu(dx)=\infty$.
\end{enumerate}
\end{enumerate}
\end{remark}

We already know by Remark~\ref{sub} that $G^a$, $a>0$, describes $\KK(X)$ by \eqref{def:pK}. We extend this observation to $a=0$.
\begin{proposition}\label{prop:sub_G0}
Assume {\rm(S1)} and {\rm(S2)} with $a=0$. Then $\pK^0(X)=\pK(X)=\KK(X)$, that is $q\in\KK(X)$ if and only if
\begin{align*}
\lim_{r\to 0^+}\left[ \sup_{x\in\R}\int_0^{r} |q(z+x)|G^0(z)dz\right]=0\,.
\end{align*}
\end{proposition}
\begin{proof}
Obviously $X$ is transient and by Remark \ref{special_to_zero} the result of Theorem \ref{prop:G_unb_b} applies.
\end{proof}

\begin{lemma}\label{lem:osz_Ga}
Assume {\rm(S1)}, {\rm(S2)} and {\rm(S3)} and let $a\geq 0$ be chosen according to {\rm(S2)}. Then the density $G^a(z)$ of $G^a(dz)$  satisfies
$$
G^a(z)\approx \frac{\phi'(z^{-1})}{z^2\phi^2(z^{-1})},\quad 0<z\leq 1.
$$
\end{lemma}
\begin{proof}
The Laplace transform of $G^a(z)$ is given by $\Phi=1/[a+\phi]$.
Note that
$$
\Phi ' = \frac{\phi '}{\phi^2} \left[\frac{\phi}{a +\phi} \right]^2 \approx \frac{\phi '}{\phi^2} \qquad \mbox{on}\qquad [1,\infty).
$$
Thus by \cite[Remark 3]{BGR1} $\Phi ' \in {\rm WUSC}(-\beta,\ut\vee 1,\uC/c)$, $c=[\phi(1)/[a+\phi(1)]]^2$.
Next, \cite[Lemma 5]{BGR1} and a version of Lemma 13 from  \cite{BGR1} imply $G^a(z)\approx z^{-2}\Phi'(z^{-1})\approx z^{-2} \phi'(z^{-1})/\phi^2(z^{-1})$ as $z\to 0^+$ (see also \cite[Proposition 3.4]{MR2928720}). The result extends to $z\in(0,1]$ by the  regularity of both sides of the estimate.
\end{proof}

Lemma~\ref{lem:osz_Ga},
Remark~\ref{sub} and Proposition~\ref{prop:sub_G0} imply  the following result.
\begin{proposition}
Let $X$ be a subordinator satisfying {\rm (S1)}, {\rm (S2)} and {\rm (S3)}. Then $q\in\KK(X)$ if and only if
\eqref{eq:sub_Kato} holds.
\end{proposition}

\subsection{Examples}

We refer the reader to \cite{MR644024}, \cite{MR1054115}, \cite{MR1132313} and \cite{MR2345907} for basic examples of the Brownian motion, the relativistic  process,  symmetric $\alpha$-stable processes and relativistic $\alpha$-stable processes.
We proceed towards our examples.

\begin{example}\label{example:1}
Denote
$A_1=\{ 2^n\colon  n\in \ZZ \}$
and
$$
f(s)=\ind_{(0,1]}(s)\, s ^{-\alpha} + e^m \ind_{(1,\infty)}(s)\, e^{-m s^{\beta}} s^{-\delta}\,,\quad s>0\,,
$$
where $m>0$, $\beta\in(0,1]$, $\delta>0$ and $\alpha \in (0,2)$.
Define a L{\'e}vy measure in $\R$ as
\begin{align}\label{nu_disc}
\nu(dz)=\sum_{y\in A_1} f(|y|)\, \big( \delta_{y}(dz)+\delta_{-y}(dz)\big)\,.
\end{align}
Let $X$ be a L{\'e}vy process with $A=0$, $\gamma=0$ and (an infinite  symmetric) $\nu$ given by \eqref{nu_disc}.
Then $X$ is a recurrent process, $\psi(z)$ is a real valued function comparable with $|z|^2 \land |z|^{\alpha}$ (see \cite[Example 4]{2014arXiv1403.0912K} and \cite[Corollary 37.6]{MR1739520}). Further, if  $\alpha\in(1,2)$ Theorem \ref{thm:0regH0} applies and describes both $\pK(X)$ and $\KK(X)$. If now $\alpha\in (0,1]$ by Theorem \ref{prop:KatoGeneral} we obtain $\pK(X)=\KK(X)$.
By \cite[Theorem 2.5]{MR3139314} there are constants $c_1, c_2\in(0,1)$ such that $p(t,x)\geq c_1 t^{-1/\alpha}$ on $|x|\leq c_2\, t^{1/\alpha}$, $t\in (0,1]$. Then for some $c>0$
$$
\int_0^1 p(u,x)\,du\geq c\, H(|x|)\,,  \qquad |x|\leq  c_2/2\,.
$$
where
$$
H(r)=
\begin{cases}
 r^{\alpha-1}, & 0<\alpha<1, \\
 \ln(r^{-1}), & \alpha=1.
\end{cases}
$$
Moreover, by \cite[Example 4]{2014arXiv1403.0912K} there is $c_3>0$ so that $p(t,x)\leq c_3\, t^{-1/\alpha} (1\land t\,|x|^{-\alpha})$ on $|x|\leq 1$, $t\in(0,1]$.
Thus, if $\alpha\in (1/2,1]$, there exists a constant $c>0$ such that
$$
\int_0^1 p(u,x)\,du\leq c\,H(|x|)\,, \quad\qquad |x|\leq 1/2 \,.
$$
Finally, by Proposition \ref{pK_alternat} for $\alpha\in (1/2,1]$ we have $q\in\pK(X)=\KK(X)$ if and only if
$$
\lim_{r\to 0^+} \int_{B(x,r)} |q(z)| H(|z-x|)dz=0\,.
$$

We note that this considerations superficially resemble the results of \cite{MR2345907} (see especially \cite[Definition~3.2]{MR2345907}). We explain why \cite{MR2345907} cannot be applied in this example if $\alpha\leq 1$.
Let $f(t,x)$ be a function that is non-increasing on $x\in (0,1]$ for every fixed $t\in (0,1]$. If $p(t,x)\leq f(t,x)$ by the lower bound for $p$ and monotonicity of $f$ we have $f(t,x)\geq c_4 \, t^{-1/\alpha}(1 \land t\,2^{\alpha k})$, $x\in(2^{-k-1},2^{-k}]$. Then for $n(t)=(1/\alpha)\log_2(1/t)$ we obtain
$$
\int_0^1 f(t,x)dx \geq c_4\, t^{1-1/\alpha}\sum_{k=0}^{n(t)} 2^{(\alpha-1)k-1} \quad \xrightarrow {t\to 0^+} \infty\,, \quad {\rm if}\quad  \alpha \in(0,1].
$$
Finally, if the upper bound assumption \cite[(A2.3)]{MR2345907} holds, i.e., $p(t,x)\leq t^{-1/\beta} \varPhi_2(t^{-1/\beta}|x|)=f(t,x)$ for some $\beta>0$, we have
$$
\int_0^{t^{-1/\beta}} \varPhi_2(z)dz =\int_0^{1} f(t,x)dx \quad \xrightarrow {t\to 0^+} \infty\,, \quad {\rm if}\quad  \alpha \in(0,1]\,,
$$
which contradicts with the integrability assumption in \cite[(A2.3)]{MR2345907}.

In fact, we have $p(s,x)\leq c_3 \,t^{-1/\alpha} \varPhi_2 (t^{-1/\alpha}|x|)$ for $|x|\leq 1$, $t\in (0,1]$ with $\varPhi_2(r)=1\land r^{-\alpha}$, which is a precise estimate for $x\in A_1$ and $|x|\leq 1$, and the integrability condition for $\varPhi_2$ holds only if $\alpha\in (1,2)$
\end{example}

\begin{example}
Let $\psi(x,y)=|x|^2+iy$ that is $X_t=(B_t,t)$, where $B_t$ is the standard Brownian motion in $\Rd$ (see \cite[10.4 and Example 13.30]{MR0481057}).  We note that in this case the transition kernel is not absolutely continuous  but the potential kernel is. Then $q\in\KK(X)$
reads as
$$
\lim_{t\to 0^+}\sup_{x\in\Rd,\,y\in\R}\int_0^t\int_{\Rd} |q(z+x,u+y)|\, u^{-d/2}e^{-|z|^2/(4u)}dzdu=0\,,
$$
and by Corollary~\ref{cor:time} holds if and only if
$$
\lim_{r\to 0^+}\sup_{x\in\Rd,\,y\in\R}\int_0^r\int_{B(0,r)} |q(z+x,u+y)|\, u^{-d/2}e^{-|z|^2/(4u)}dzdu=0\,.
$$
\end{example}

Now we discuss in detail subordinators. Since functions $\phi$ presented below are unbounded CBF with zero drift term, see \cite[Chapter 16: No 2 and 59, Proposition 7.1]{MR2978140}, they satisfy (S1) and (S2). The assumption (S3) can be easily checked.
The first example covers the case of $\alpha$-stable subordinator, $\alpha\in (0,1)$, and the inverse Gaussian subordinator.
\begin{example}
Let $\phi(u)=\delta[(u+m)^{\alpha}-m^{\alpha}]$, $\delta>0$, $m\geq 0$, $\alpha\in (0,1)$.
Then $q\in\KK(X)$ if and only if
$$
\lim_{r\to 0^+}\sup_{x\in\R}\int_x^{x+r} |q(z)| (z-x)^{\alpha-1}\, dz=0\,.
$$
\end{example}
\begin{example}
Let $\phi(u)=\ln(1+u^{\alpha})$, where  $\alpha\in (0,1]$.  Then $q\in\KK(X)$ if and only if
$$
\lim_{r\to 0^+}\sup_{x\in\R}\int_x^{x+r} |q(z)| \frac{dz}{(z-x)\ln^2(z-x)}=0\,.
$$
\end{example}
\begin{example}
Let $\phi(u)=\dfrac{u}{\ln(1+u^{\alpha})}$, where  $\alpha\in (0,1)$.  Then $q\in\KK(X)$ if and only if
$$
\lim_{r\to 0^+}\sup_{x\in\R}\int_x^{x+r} |q(z)| |\ln (z-x)|dz=0\,.
$$
\end{example}

\section*{Acknowledgement}
We thank Krzysztof Bogdan, Kamil Kaleta, Mateusz Kwa{\'s}nicki, Moritz Kassmann and Micha{\l} Ryznar for discussions, many helpful comments and references.

\bibliographystyle{abbrv}

\end{document}